\RequirePackage{fix-cm}
\documentclass[smallextended]{svjour3}       
\smartqed  
\usepackage{graphicx}
\usepackage{amsmath}
\usepackage[T1]{fontenc}
\usepackage{mathrsfs}
\usepackage{latexsym}
\usepackage{amssymb}
\usepackage{graphicx}
\usepackage{float}
\usepackage{extarrows}
\usepackage{epstopdf}
\usepackage{caption}
\usepackage{subcaption}
\usepackage{amsmath}
\usepackage{wrapfig}
\usepackage[colorlinks]{hyperref}
\usepackage{lipsum}
\usepackage{algpseudocode}
\usepackage{algorithm}
\usepackage{subcaption}
\usepackage{multicol}
\DeclareCaptionSubType*{algorithm}

\DeclareCaptionLabelFormat{alglabel}{Alg.~#2}

\begin{document}

\title{Krylov projection methods for linear Hamiltonian systems}

\author{Elena Celledoni         \and
        Lu Li 
}

\institute{Elena Celledoni  \at
              Department of Mathematical Sciences, NTNU, 7491 Trondheim, Norway \\
              Tel.: +47 73593541\\
              \email{elena.celledoni@ntnu.no}           
           \and
           Lu Li\at
           Department of Mathematical Sciences, NTNU, 7491 Trondheim, Norway \\
            Tel.: +47 73591650\\
            \email{lu.li@ntnu.no} 
}

\date{Received: date / Accepted: date}

\maketitle

\begin{abstract}
We study geometric properties of Krylov projection methods for large and sparse linear Hamiltonian systems. We consider in particular energy-preservation.  We discuss the connection to structure preserving model reduction. We illustrate the performance of the methods by applying them to Hamiltonian PDEs.
\end{abstract}

\section{Introduction}
\label{sec:introduc}

Large and sparse linear Hamiltonian systems arise in many fields of science and engineering, examples are models in network dynamics \cite{van2014port} and the semi-discretization of Hamiltonian partial differential equations (PDEs), like the wave equation \cite{feng1987symplectic,mclachlan1993symplectic}  
and Maxwell's equations \cite{marsden1982hamiltonian,sun10sam}. 
In the context of Hamiltonian PDEs, the energy conservation law often plays a crucial role in the proof of existence and uniqueness of solutions \cite{taylor115partial}. Energy-preservation under numerical discretization can be advantageous as it testifies correct qualitative behaviour of the numerical solution, and it is also useful to prove convergence of numerical schemes  
\cite{richtmyer1967difference}. 
There is an extensive literature on energy-preserving methods for ordinary differential equations (ODEs) \cite{labudde1975energy,mclachlan1999geometric,brugnano10hbv,celledoni2012preserving},
 but these methods need to be implemented efficiently to be competitive for large and sparse systems arising in numerical PDEs. Krylov projection methods are attractive for discrete PDE problems because they are iterative, accurate and they allow for restart and preconditioning strategies. But their structure preserving properties are not completely understood and should be further studied.
  
 It is well known that integration methods cannot be simultaneously symplectic and energy-preserving on general Hamiltonian systems \cite{ge88lph}. However, the situation changes when we restrict to linear systems. An example is the midpoint rule which is symplectic and is also energy-preserving on linear problems because it coincides with the AVF method \cite{quispel08anc}. The midpoint method is implicit and requires the solution of one linear system of algebraic equations at each time step. The structure preserving properties are then retained only to the precision of the linear iterative solver.
In this paper, we investigate preservation of geometric properties in Krylov projection methods. These are attractive methods for the solution of large systems arising in PDEs \cite{botchev2009numerical}, but because of the Krylov projection, symplecticity is only preserved to the accuracy of the method. On the other hand, we show that some of these methods can be energy-preserving to a higher level of precision, and can preserve several first integrals simultaneously.
We finally discuss the connections to structure-preserving model reduction and variational principles.
Previous work in the context of structure preserving Krylov projection methods can be found in \cite{lopez2006preserving,MR3126684} and for Hamiltonian eigenvalue problems for example in \cite{benner98ans}.

The structure of this paper is as follows. We discuss symplecticity in section \ref{Krylov projection and symplecticity}. Section~\ref{sec:APM} is devoted to the preservation of first integrals.
Section \ref{sec:Arnoldi-Model reduction} is devoted to projection methods based on block $J$-orthogonal bases and their connection to structure preserving model reduction. In Section \ref{sec:numerical simulation}, the geometric properties of the considered methods are illustrated by numerical examples.

\section{Krylov projection and symplecticity}\label{Krylov projection and symplecticity}
Consider a linear Hamiltonian initial value problem of the form
\begin{equation}\label{original eq}
\dot{y}=JH\,y,\quad y(0)=y_{0},\hskip2cm
J=J_m = \left[ \begin{matrix}
0 & I_{m}\\
-I_{m} &  0 
\end{matrix} \right],
\end{equation} 
\noindent where $y(t)\in \mathbb{R}^{2m}$, $H\in \mathbb{R}^{2m\times 2m}$ is symmetric, $y_{0}\in \mathbb{R}^{2m}$, and $I_m$ is the $m\times m$ identity matrix. In what follows we denote by $A$ the product $A=JH$. The skew-symmetric matrix $J$ defines a symplectic inner product on $\mathbb{R}^{2m}$, $\omega (x,y):=x^TJy.$
The vector field of equation  \eqref{original eq} is a Hamiltonian vector field.
The flow  of a Hamiltonian vector field is a symplectic map. This means that 
$\varphi_t:\mathbb{R}^{2m}\rightarrow \mathbb{R}^{2m}$, $y_0\mapsto y(t)$,
is such that 
$\Psi_{y_0}(t)=\frac{\partial \varphi_t(y_0)}{\partial y_0}$ satisfies 
\begin{equation}
\label{symplecticity}
\Psi_{y_0}(t)^T\, J \, \Psi_{y_0}(t) = J.
\end{equation}
In other words, $\Psi_{y_0}(t)$ is an element of the symplectic group $\mathrm{Sp}(2m)$, and
$y(t)=\Psi_{y_0}(t)y_0$. Crucially, any element of $U\in \mathrm{Sp}(2m)$ is such that the change of variables $x=U y$ sends Hamiltonian systems to Hamiltonian systems.
Denote by 
$\mathcal{H}(y)=\frac{1}{2}{y}^{T}J^{-1}Ay $ 
 the energy function. 
 Another fundamental property of the system  (\ref{original eq}) is that $\mathcal{H}$ is constant along solution trajectories, i.e.,
$\frac{d\, \mathcal{H}(y(t))}{d t}=0.$ An approximation method for \eqref{original eq} is said to be energy-preserving if  $\mathcal{H}$ is constant along the numerical solution, and symplectic if the numerical flow  $\phi_h: \mathbb{R}^{2m}\rightarrow \mathbb{R}^{2m}$, $y_0\mapsto \tilde{y}$ with $\tilde{y}\approx y(h)$, is such that $\frac{\partial \phi_h(y_0)}{\partial y_0}^TJ\frac{\partial \phi_h(y_0)}{\partial y_0}=J.$

The idea of Krylov projection methods is to build numerical approximations for (\ref{original eq})  in the Krylov subspace: 
$$\mathcal{K}_{r}(A, y_{0}):=\mathrm{span} \{y_{0}, Ay_{0}, \cdots ,A^{r-1}y_{0}\},$$   
which is a subspace of $\mathbf{R}^{2m}$ of dimension $r<<2m$.
Let us consider even dimension $r=2n$. A basis of $\mathcal{K}_{2n}(A, y_{0})$ is constructed. The most well known Krylov projection method is the one based on the Arnoldi algorithm \cite{arnoldi1951principle}  generating an orthonormal basis for $\mathcal{K}_{2n}(A, y_{0})$. The method gives rise to a $2m\times 2n$ matrix $V_{2n}$ with orthonormal columns, and to an upper Hessenberg $2n\times 2n$ matrix $H_{2n}$ such that
$I_{2n}={V_{2n}}^{T}V_{2n},$ and $H_{2n}=V_{2n}^TAV_{2n}.$
The approximation of $y(t)$ is 
\begin{equation}
y_A:=V_{2n}z(t),\quad \mathrm{where} \quad \dot{z}=H_{2n}\,z, \quad z(0)=z_0=V_{2n}^Ty_{0}.
\label{Arnoldi small system}
\end{equation}
We will denote this method by Arnoldi projection method (APM). Consider $J_{2n}$ and the symplectic inner product in $\mathbb{R}^{2n}$, $\bar{\omega}(\bar{x},\bar{y})=\bar{x}^TJ_{2n}\bar{y}$.
If $n<m$, unless we make further assumptions on $H$,    the projected system \eqref{Arnoldi small system} is not a Hamiltonian system in $\mathrm{R}^{2n}$, this can be seen because $J_{2n}^{-1}H_{2n}=J_{2n}^{-1}V_{2n}^TJHV_{2n}$ is in general not symmetric.

Instead of using an orthonormal basis, one can construct a
$J$-orthogonal basis for $\mathcal{K}_{2n}(A, y_{0})$ using the symplectic Lanczos algorithm \cite{benner2011hamiltonian}. 
The matrix 
$S_{2n}$ whose columns are the vectors of this $J$-orthogonal basis satisfies
$$S_{2n}^TJS_{2n}=J_{2n}.$$ 
We will denote the corresponding Krylov projection method by Symplectic Lanczos projection method (SLPM).
The projected system for SLPM is analog to \eqref{Arnoldi small system}, with $V_{2n}$ replaced by $S_{2n}$,  $H_{2n}$ by $J_{2n}S_{2n}^THS_{2n}$
and  an appropriate $z_0$ (see Section~\ref{subsec:SLPM}). 
This projected system is a Hamiltonian system. 
But for $n<m$, the approximation $y_S(t):=S_{2n}z(t)$ is not symplectic.  In fact, $y_S$ is the solution of the system 
\begin{equation}
\label{SLPM equation}
\dot{y}_S=(S_{2n}J_{2n}S_{2n}^T)\,H\,y_S,\quad y_S(0)=y_0,
\end{equation}
which is a Poisson system with Poisson structure given by the skew-symmetric matrix $(S_{2n}J_{2n}S_{2n}^T)$ which depends on the initial condition\footnote{A Poisson system in $\mathbb{R}^{d}$ is a system of the type $\dot{y}= \Omega\, \nabla \mathcal{H}(y)$, where $\Omega$ is skew-symmetric, not necessarily invertible and can depend on $y$. In our case, $\Omega$ depends on $y_0$.}.
For $n=m$, $S_{2m}\in\mathrm{Sp}(2m)$, $J_{2m}=J$, and $y_S=y.$ 
However, the case $n<m$ is the most relevant for the use of the method in practice.
In spite of not preserving $\omega$, 
SLPM clearly shares important  structural properties with the exact solution of \eqref{original eq} and is energy-preserving, see Section~\ref{subsec:SLPM}.

The symplectic Lanczos algorithm is not the only way to obtain a $J$-orthogonal basis of the Krylov subspace. We will consider block $J$-orthogonal bases in Section~\ref{sec:Arnoldi-Model reduction} and show that they can be viewed as techniques of structure preserving model reduction, in the spirit of \cite{lall2003structure}. We propose one Krylov algorithm based on these ideas.

\section{Preservation of first integrals and energy}\label{sec:APM}
We first present a result about the first integrals for a general linear Hamiltonian system. 
\begin{proposition}\label{first integrals original eq}
For $A=JH$ where $J$ is skew symmetric and invertible, and $H$ is symmetric and invertible, the system $\dot{y}=Ay$, $y(0)=y_0$ has $m$ independent first integrals in involution, $\mathcal{H}_k(y)=\frac{1}{2} \langle y, A^{2k}y\rangle_H$ for $k=0,1,\dots ,m-1$. The Hamiltonian of the system is $\mathcal{H}=\mathcal{H}_0$.
\end{proposition}
\begin{proof}
We have
\begin{eqnarray*}
\frac{d}{dt}\mathcal{H}_k(y)&=&\frac{1}{2}\left[\dot{y}^TH(JH)^{2k}y+y^TH(JH)^{2k}\dot{y}\right]\\
&=&
\frac{1}{2}\left[-y^THJH(JH)^{2k}y+y^TH(JH)^{2k}JHy\right]\\
&=&\frac{1}{2}\left[-y^TH(JH)^{2k+1}y+y^T H(JH)^{2k+1}y\right]=0,
\end{eqnarray*}
so $\mathcal{H}_k$, $k=0,\dots , m-1$ are preserved along solutions of $\dot{y}=Ay$, $y(0)=y_0$.
The integrals are in involution because their Poisson bracket is zero,
\begin{align*}
\{\mathcal{H}_k,\mathcal{H}_p\}&=(\nabla \mathcal{H}_k)^T J\nabla \mathcal{H}_p=y^T((JH)^{2k})^THJH(JH)^{2p}y\\
&=y^TH(JH)^{2(k+p)+1} y=0,
\end{align*}
where we have used the skew-symmetry of $H(JH)^{2(k+p)+1}$. 
The integrals are functionally independent because, when $J$ and $H$ are invertible, $J^{-1}\nabla \mathcal{H}_k=2J^{-1}A^{2k}y$ for $k=0,\dots, m-1$ are linearly independent vectors\footnote{Note that the invertibility of $J$ and $H$  is needed only to prove that the integrals are functionally independent in this proof.}
\end{proof}
In what follows, we will discuss the preservation of the first integrals of Proposition~\ref{first integrals original eq}  when applying Krylov projection methods.

\subsection{Preservation of first integrals for the APM}
It can be observed from numerical simulations
that the APM fails in general to preserve energy when applied to Hamiltonian systems, Figure~\ref{Block diagonal-energy error-apm}, Section~\ref{sec:numerical simulation}, but structure-preserving properties can be ensured for such method via a simple change
of inner product. Assume that $H$ is symmetric and positive definite so that $\langle \cdot, \cdot \rangle_H:=\langle \cdot, H\cdot \rangle$ defines an inner product.
We modify the Arnoldi algorithm by replacing the usual inner product $\langle \cdot, \cdot \rangle$ by $\langle \cdot, \cdot \rangle_H$. We then show that the numerical solution given by this method preserves to machine accuracy certain first integrals.
The modified Arnoldi algorithm  (see Algorithm \ref{alg:Arnoldi modified inner product algorithm}) 
generates a $H$-orthonormal basis, which is stored in the $2m\times n$ matrix $V_{n}$, satisfying $V_{n}^{T}HV_{n}=I_{n}$.  This algorithm generates an upper Hessenberg matrix $H_{n}$ such that 
\vspace{-7pt}
\begin{equation*}\label{modified Arnoldi relation}
\begin{split}
AV_{n}&=V_{n}H_{n}+w_{n+1}e_{n}^{T},\quad w_{n+1}=h_{n+1,n}v_{n+1},\\
{V_{n}}^{T}HV_{n}&=I_{n}, \quad {V_{n}}^{T}Hw_{n+1}=0.
\end{split}
\vspace{-10pt}
\end{equation*}
In what follows, we consider the Krylov projection method 
$$y_H:=V_{n}z,\quad \mathrm{where}\,\, z\,\, {\mathrm{satisfies}} \quad 
\dot{z}=H_{n}\,z,\quad
z(0)={V_{n}}^{T}Hy_0.
$$
\begin{proposition}\label{first integral projected inner original Equations}
The numerical approximation $y_H$ for the solution $y$ of \eqref{original eq}
preserves the following first integrals: 
\begin{equation}
\label{DAPM}
\mathcal{H}_{k}(y_H)=\frac{1}{2}y_H^{T}HV_{n}(H_{n})^{2k}{V_{n}}^{T}Hy_H
\end{equation}
for all $k=0, 1, \dots, r$, where  
$r=[\frac{n}{2}]-1$. 
\end{proposition}
\vspace{-10pt}
\begin{proof}
We observe that $H_{n}={V_{n}}^{T}HJHV_{n}$ is skew-symmetric. So the ODE system for $z$ has first integrals: $\mathcal{I}_{k}(z)= \frac{1}{2}z^{T}(H_{n})^{2k}z$, for all $k=0, 1, \dots, r$ with $r=n/2-1$ if $n$ is even and $r=(n-1)/2-1$ if $n$ is odd. 
Therefore $\mathcal{H}_{k}(y_H)=\frac{1}{2}y_H^{T}HV_{n}(H_{n})^{2k}{V_{n}}^{T}Hy_H=\frac{1}{2}z^{T}(H_{n})^{2k}z$
 are preserved.
\end{proof}
\begin{remark}
If $n$ is even, the above Krylov projection method induces a projected problem which is conjugate to a Hamiltonian system, i.e., it can be written in the form \eqref{original eq} via  change of variables. Since $H_n$ is skew-symmetric, $H_n$ can be factorized as
$H_n=U_nJ_nD_nU_n^{-1}$ where $D_n$ is diagonal. Then, $H_n$ can be transformed to a Hamiltonian matrix by a similarity transformation using $U_n$.
\end{remark}

\subsection{Hamiltonian system with $JA=AJ$}\label{case $A$ Hamiltonian and $JA=AJ$}
We now consider $J$ given  by \eqref{original eq}. Assume that $A$ and $J$ commute, then $A$ is skew-symmetric, and the Hamiltonian system \eqref{original eq} has two Hamiltonian structures, one associated to $A$ with Hamiltonian $\frac{1}{2}y^Ty$, the other to $J$ with Hamiltonian $\frac{1}{2}y^THy$. The APM with Euclidean inner product $\langle \cdot, \cdot \rangle$ preserves modified first integrals. To proceed, we first give without proof the following result. 
\begin{proposition}\label{prop:1}
Suppose $A$ is a Hamiltonian matrix. Then $J$ and $A$ commute if and only if the matrix $A$ is skew-symmetric. 
\end{proposition} 
The first integrals of the  system \eqref{original eq} are given by the following proposition.
\begin{proposition}\label{first integrals original eq commutativity}
If $JA=AJ$, the Hamiltonian system \eqref{original eq} has $m$ independent first integrals in involution, $\mathcal{H}_k(y)=\frac{1}{2} y^TA^{2k}y$ for $k=0,1,\dots ,m-1$, and in involution with the Hamiltonian $\mathcal{H}(y)=\frac{1}{2}y^THy$.
\end{proposition}
\begin{proof}
From Proposition~\ref{prop:1} we know that $A$ is skew-symmetric. Then  Proposition~\ref{first integrals original eq} holds with $J$ replaced by $A$, and $H$ replaced by the identity matrix.
The integrals are in involution with the Hamiltonian $\mathcal{H}(y)=\frac{1}{2}y^THy$ in fact using the commutativity of $A$ and $J$
$$\{\mathcal{H}_k,\mathcal{H}\}=y^TA^{2k}JAy=y^TA^k(JA)A^ky=0,\quad k=0,\dots, m-1.
$$ 
\end{proof}
\begin{remark}\label{invariants for Arnoldi projection} 
By a direct application of Proposition~\ref{first integral projected inner original Equations}, the APM to the Hamiltonian system \eqref{original eq}, under the assumption $JA=AJ$, gives a numerical approximation $y_A:=V_nz$ which preserves the following modified first integrals
$$\mathcal{\tilde{H}}_k(y_A):=\frac{1}{2}y_A^TV_n(H_n)^{2k}V_n^Ty_A, \quad k=0,1,\dots, n.$$
\end{remark} 
We next prove that the Hamiltonian  of \eqref{original eq} is bounded by   $y_A$ under the assumption that $J$ and $A$ commute.
\begin{proposition}\label{bounded energy}
Assume the APM is applied to \eqref{original eq}.
Under the assumption $JA=AJ$, the energy 
$\mathcal{H}(y)=\frac{1}{2}{y}^{T}J^{-1}Ay,$
is bounded along the numerical solution.
\end{proposition} 
\begin{proof}
This result follows directly from Remark~\ref{invariants for Arnoldi projection} with $k=0$, i.e.,
 \begin{equation*}
 \frac{1}{2}{y_{A}}^{T}J^{-1}Ay_{A}\le\frac{1}{2}{y_{A}}^{T}y_{A}\|J^{-1}A\|_{2} =\frac{1}{2}{y_{0}}^{T}y_{0}\|J^{-1}A\|_{2}.
\end{equation*} 
\end{proof} 
Proposition~\ref{bounded energy} explains the good behaviour of the APM  in  \cite{celledoni16epk}.
\subsection{Symplectic Lanczos projection method}\label{subsec:SLPM}
 We now introduce the symplectic Lanczos projection method (SLPM). For this method the projected system  \eqref{Arnoldi small system} is  a Hamiltonian system. We  prove that the SLPM preserves the energy of the original system.

Given  $A\in R^{2m,2m}$ and the starting vector $y_{0}\in R^{2m}$,  the symplectic Lanczos method generates a sequence of matrices
\begin{equation}\label{symplectic-lancozs symplectic form}
S_{2n}=[v_{1},...,v_{n},w_{1},...w_{n}]\quad\mathrm{satisfying} \quad AS_{2n}=S_{2n}H_{2n}+r_{n+1}e^{T}_{2n},
\end{equation}
where $H_{2n}$ is a tridiagonal Hamiltonian matrix, and $r_{n+1}=\zeta_{n+1}v_{n+1}$ is $J$-orthogonal with respect to the columns of $S_{2n}$. Since $S_{2n}$ has $J$-orthogonal columns, i.e., ${S_{2n}}^{T}JS_{2n}=J_{2n}$, we know that
\begin{equation}\label{symplectic-lancozs H2n}
H_{2n}=J_{2n}^{-1}\,S_{2n}^TJAS_{2n}=J_{2n}S_{2n}^THS_{2n},
\end{equation}
and the projected system 
is a Hamiltonian system, where $z_{0}=J_{2n}^{-1}S_{2n}^TJy_{0}$. Moreover, we have
\begin{equation}\label{SLPM small system-energy}
\mathcal{H}_{S}(z)=\frac{1}{2} z^{T}J_{2n}^{-1}H_{2n}z\equiv\frac{1}{2} z_{0}^{T}J_{2n}^{-1}H_{2n}z_{0}.
\end{equation}

\begin{proposition}\label{SLPM energy preservation}
The SLPM is an energy-preserving method for  \eqref{original eq}.
\end{proposition}
\begin{proof}
 The result follows by computing the Hamiltonian of \eqref{original eq} along numerical trajectories $y_S=S_{2n}z$, $\mathcal{H}(y_S)=\frac{1}{2}{y_S}^{T}J^{-1}Ay_S,$ and then using \eqref{symplectic-lancozs symplectic form} and \eqref{SLPM small system-energy}.
 \end{proof}
 \section{Projection methods based on block $J$-orthogonal basis} 
\label{sec:Arnoldi-Model reduction}
We now consider a general strategy for Krylov projection methods to obtain  $J$-orthogonal bases.
In what follows we will use the notation $(q^T,p^T)^{T}=y$ 
and write $H$ in block form, and rewrite \eqref{original eq} accordingly:
\begin{equation}
\begin{split}\label{Hamiltonian equation pq}
\dot{q}&=H_{12}^{T}q+H_{22}p,\\
\dot{p}&=-H_{11}q-H_{12}p,
\end{split} \qquad \quad
H=
\begin{bmatrix}
     H_{11}& H_{12} \\
H_{12}^{T}& H_{22}
\end{bmatrix}.
\end{equation}
Assume that we can construct two matrices with linearly independent columns $V_n\in {\mathbb{R}}^{m\times n}$ and $W_n\in {\mathbb{R}}^{m\times n}$ such that ${V_{n}}^{T}W_{n}=I_{n}$. 
Then the matrix 
\begin{equation}
\label{blockSLPbasis}
S_{2n}:=\begin{bmatrix}
 V_{n} & 0 \\
0& W_{n}
\end{bmatrix}
\end{equation}
has $J$-orthogonal columns. We will approximate $y$ by the following projection method: $y\approx y_B$ defined by
\begin{equation}
\label{blockSLPM}
y_B= S_{2n}\,z, \quad \mathrm{where}\,\, z\,\, {\mathrm{satisfies}} \quad \dot{z}=J_{2n}\,S_{2n}^TJ^{-1}AS_{2n}, \quad z(0)=z_0,
\end{equation}
and for the SLPM $z_{0}=J_{2n}^{-1}S_{2n}^TJy_{0}$.

\begin{proposition}\label{Energy preservationof BJPMs}
If $y_0=S_{2n}z(0)$, then the energy   of the original Hamiltonian system \eqref{original eq} will be preserved by the numerical solution  \eqref{blockSLPbasis}-\eqref{blockSLPM}.
\end{proposition}
\begin{proof}
Notice that 
$\mathcal{H}(S_{2n}\,z)=\frac{1}{2}  z^T S_{2n}^TJ^{-1}A S_{2n}\,z$
 is a constant because $z$ is the solution of  a Hamiltonian system with energy $\mathcal{K}(z)=\frac{1}{2}  z^T (S_{2n}^TJ^{-1}A S_{2n})\,z$. The result then follows directly from the fact that $\mathcal{K}(z)\equiv\mathcal{K}(z_0)=\mathcal{H}(y_0)$.
\end{proof}

We here propose one strategy to construct $S_{2n}$ as in \eqref{blockSLPbasis} with $W_n^TV_n=I_n$ and  $V_n=W_n$.
Let $K_n$ be the Krylov matrix $2m\times n$, and consider the first $m$ rows of $K_n$ and the last $m$ separately: 
$$K_n:=[y_0,Ay_0, \dots ,A^{n-1}y_0], \qquad K_n=\left[ \begin{array}{c}
K_n^q\\[0.2cm]
K_n^p
\end{array}
 \right].$$
We then find an orthonormal  basis $V_n$ for $\mathrm{span}\{K_n^q,K_n^p\}\subset \mathbb{R}^m$ by either a QR-factorisation (algorithm~\ref{alg:orthogonal_Vn} in the Appendix \footnote{Notice that to obtain a stable algorithm it is an advantage to replace the Krylov matrix with an orthonormal matrix obtained by the Arnoldi algorithm.}) or a Gram-Schmidt process.
\subsection{Structure preserving model reduction using Krylov subspaces}\label{Case V_n V_n only}

In this section we consider the  variational principle lying behind the presented techniques. This allows to draw connections to the techniques of structure preserving model reduction of \cite{lall2003structure}, see also \cite{peng2016symplectic}.
Assuming additional structure for $H$, we will also show that the usual APM applied to the resulting system coincides with  a structure preserving model reduction method.

Assume $[q^T,p^T]^T:=y$ and $q$ and $p$ are $m$-dimensional vectors belonging to $\mathbf{R}^m$ and its dual respectively, and that the Hamiltonian $\mathcal{H}:\mathbf{R}^m\times (\mathbf{R}^m)^*\rightarrow \mathbf{R}$ is $\mathcal{H}(q,p):=\mathcal{H}(y).$\footnote{The duality pairing between $\mathbf{R}^m$ and $(\mathbf{R}^m)^*$ is here simply $\langle p,q\rangle:=p^Tq$.}
Considering the action functional $\mathcal{S}:\mathbf{R}^m\times (\mathbf{R}^m)^*\rightarrow \mathbf{R}$ 
\begin{equation*}
 \mathcal{S}(q,p):= \int_ {t_{0}}^{t_{end}}\left( {p(t)}^{T}\dot{q}(t)-\mathcal{H}(q(t),p(t)) \right)\,dt,
\end{equation*}
Hamilton's phase space variational principle states that
$$\delta \mathcal{S} =0$$
for fixed $q_0=q(t_0)$ and $q_{end}=q(t_{end})$, and it is equivalent to Hamilton's equations \eqref{original eq}.
By projecting $q(t)$ and $p(t)$ separately on appropriate subspaces  $\mathrm{span}\{V_n\}\subset\mathbf{R}^m$ and $\mathrm{span}\{W_n\}\subset(\mathbf{R}^m)^*$, i.e., $q(t)\approx V_n\,\hat{q}(t)$ and $p(t)\approx W_n\,\hat{p}(t)$, one restricts the variational principle to $\mathrm{span}\{V_n\}\times \mathrm{span}\{W_n\}$:
$\hat{\mathcal{S}}(\hat{q},\hat{p}):=\mathcal{S}(V_n\,\hat{q}, W_n\,\hat{p}).$
By taking variations
\begin{equation*}
0=\delta \hat{\mathcal{S}}(\hat{q},\hat{p})= \delta\int_ {t_{0}}^{t_{end}}  (V_{n}\hat{p})^{T}W_{n}\dot{\hat{q}}(t)-H(V_{n}\hat{p},W_{n}\hat{q}) dt
\end{equation*}
for fixed endpoints $\hat{q}_0=\hat{q}(t_0)$ and $\hat{q}_{end}=\hat{q}(t_{end})$,
we obtain the Hamiltonian equations associated to this reduced variational principle
\begin{equation}\label{model reduction p-q smaller prob-V-W}
\begin{split}
\dot{\hat{p}}&=-{V_{n}}^{T}{H_{12}}W_{n}\hat{p}-{V_{n}}^{T}H_{11}V_{n}\hat{q},\\
\dot{\hat{q}}&={W_{n}}^{T}H_{22}W_{n}\hat{p}+{W_{n}}^{T}H_{12}^{T}V_{n}\hat{q},
\end{split}
\end{equation}
which coincide with the system for $z$ in \eqref{blockSLPM}.

\subsection{Special case $H_{1,2}=O$, $H_{2,2}=I$.}\label{specail structure of Hamiltonian matrix}
This special case is directly related to the setting in \cite{lall2003structure}. Denoting $y=(q^{T},p^{T})^{T}$, we consider the action functional associated to the Lagrangian 
\begin{equation}
\label{special1}
L(q(t),\dot{q}(t))=\frac{1}{2} {\dot{q}(t)}^{T}\dot{q}(t)-\frac{1}{2}{q(t)}^{T}H_{11}q(t)
\end{equation}
and the corresponding Hamiltonian system
\begin{equation}\label{special2-y form}
\dot{y}=Ay \quad \text{with} \quad A=\begin{bmatrix}
    0    &  I  \\
      -H_{11}    &   0
\end{bmatrix}.
\end{equation}
Let $V_n$ be the basis of the Krylov subspace $\mathcal{K}_n(-H_{11},p_0)$ obtained via the Arnoldi algorithm.
The reduced Lagrangian becomes
\begin{equation}
\label{special3}
L(\hat{q}(t),\dot{\hat{q}}(t))=\frac{1}{2} {\dot{\hat{q}}(t)}^{T}\dot{\hat{q}}(t)-\frac{1}{2}{\hat{q}(t)}^{T}{V_{n}}^{T}H_{11}V_{n}\hat{q}(t),
\end{equation}
and the corresponding Hamiltonian equations are 
 \begin{equation}\label{model reduction q-smaller prob}
\begin{split}
\dot{\hat{q}}&=\hat{p},\\
\dot{\hat{p}}&=-{V_{n}}^{T}H_{11}V_{n}\hat{q}(t).
\end{split}
\end{equation}
By  solving \eqref{model reduction q-smaller prob}, we  obtain $(\hat{q}^T,\hat{p}^T)^{T}$ and then can construct the model reduction approximation $((V_{n}\hat{q})^T, (V_{n}\hat{p})^T)^{T}\approx (q^T,p^T)^T$. 
\begin{proposition} \label{APM-model reduction}
When applied to  \eqref{special2-y form} with  $y_{0}=(0,p_{0}^{T})^{T}$, 
the model reduction procedure outlined in \eqref{special1}-\eqref{model reduction q-smaller prob} coincides with the APM.
\end{proposition}
\begin{proof} 
Let $\mathbf{e}_1,\mathbf{e}_2\in\mathbf{R}^2$ be the two vectors of the canonical basis in $\mathbf{R}^2$. Denote by $\otimes$ the Kronecker tensor product. We have 
$$\mathcal{K}_{2n}(A,y_0)=\mathrm{span}\{ \mathbf{e}_1 \otimes p_0, \mathbf{e}_2 \otimes p_0,\mathbf{e}_1 \otimes (-H_{11}p_0),\mathbf{e}_2 \otimes (-H_{11})p_0, \dots\}.$$
Denote by $\mathbb{U}_{2n}\in {\mathbb{R}}^{2m\times 2n}$ the orthogonal matrix generated by the usual Arnoldi algorithm with matrix $A$, vector $y_{0}=(0,p_{0}^{T})^{T}$ and Euclidean inner product.  Then $\mathbb{U}_{2n}$ is given by
$$\mathbb{U}_{2n}=\left[  
\begin{array}{cccccccc}
0 & v_1 & 0 & v_2 &0 &\dots &0 & v_n\\
v_1   & 0& v_2 & 0 & v_3&\dots &v_n & 0
\end{array}
\right],
$$
and satisfies 
\begin{equation*}
{\mathbb{U}_{2n}}^{T}A\mathbb{U}_{2n}=\Pi_{2n}\begin{bmatrix}
     
 0& I_{n} \\
-V_{n}^{T}H_{11}V_{n}& 0 
\end{bmatrix}{\Pi_{2n}}^{T}\quad\mathrm{and}\quad \mathbb{U}_{2n}\Pi_{2n}=\left[
\begin{array}{cc}
V_n & O\\
O & V_n
\end{array}
\right],
\end{equation*}
where $v_1,v_2,\dots v_n $ are the columns of $V_n$ and $\Pi_{2n}$ is a $2n\times 2n$ permutation matrix.
After a permutation of the variables $w={\Pi_{2n}}^{T}z$, the projected system by APM $\dot{z}={\mathbb{U}_{2n}}^{T}A\mathbb{U}_{2n}z$, $z(0)={\mathbb{U}_{2n}}^{T}y_{0}$
can be rewritten in the form \eqref{blockSLPbasis}-\eqref{blockSLPM}.
\end{proof}

\section{Numerical Examples}\label{sec:numerical simulation}
\captionsetup[subfigure]{margin=0pt, parskip=0pt, hangindent=0pt, indention=0pt, labelformat=parens, labelfont=rm} 
In this section, several numerical examples are presented to illustrate the behavior of the methods described above. We will use the following methods:
\begin{itemize}
  \item APM: Arnoldi projection method using Euclidean inner product,  Section \ref{sec:APM};
  \item APMH: Arnoldi projection method using the inner product $\langle \cdot, \cdot \rangle_H$, Section \ref{sec:APM};
  \item SLPM: symplectic Lanczos projection method,  Section~\ref{subsec:SLPM};
  \item BJPM: block $J$-orthogonal projection method QR factorization, Section \ref{Case V_n V_n only}.
\end{itemize}
These methods are applied to solve randomly generated linear Hamiltonian
systems, and linear systems arising from the discretization of Hamiltonian
PDEs. 
\subsection{Randomly generated Hamiltonian matrices }
We consider numerical experiments on randomly generated linear Hamiltonian systems. 
If not mentioned otherwise, the dimension of the Krylov subspace is  chosen to be $2n = 4$ and is the same for all the Krylov methods compared. The reference exact solution is computed using the Cayley transformation with step-size $0.004$.  
The solution of the projected system \eqref{Arnoldi small system} is obtained with the same approach and step-size used for the reference exact solution. 
To obtain a desired global error accuracy on $[0,T]$ for large $T$, we either use a sufficiently large dimensions of the Krylov subspace or perform a restart procedure using the Krylov projection methods on sufficiently small subintervals. More precisely, the considered restart procedure consists in subdividing $[0,T]$ into subintervals $[t_k, t_{k+1}]$ and performing the projection on each subinterval recomputing the basis of the Krylov subspace with starting vector $y_k$, where $y_k$ is the numerical solution at $t=t_k$. The restart procedure is of practical interest because it allows to use a Krylov subspace of low dimension. However, the restart destroys the preservation of the first integrals of Propositions~\ref{first integral projected inner original Equations} and \ref{first integrals original eq commutativity} for APM and APMH because the basis $V_n$ is recomputed on each subinterval. Experiments comparable to the ones performed in this section can be found in \cite{peng2016symplectic} for model reduction techniques without restart.
\vspace{-10pt}
\subsubsection{Case $JA=AJ$: APM}
\vspace{-8pt}
In the experiments reported in Figure~\ref{Block diagonal-energy error-apm}, $H=J^{-1} A$ is block diagonal, symmetric and positive definite but with no particular extra structure. We observe in Figure~\ref{Block diagonal-energy error-apm} that there is a drift in the energy for the APM, and the advantage of APMH and SLPM is evident in this example. The global errors are not reported here, but we find that the global errors of APMH and SLPM are bounded, meanwhile there is a substantial drift in the global error of the APM in this case. In Figure~\ref{Block skew-energy error-apm} and \ref{Block skew-integrals-apm}, we apply the APM to an example where $JA=AJ$. The experiments confirm the good behaviour of the APM in this case. Figure~\ref{Block skew-energy error-apm} shows that the energy error for the APM is bounded even though the energy is not preserved exactly. The global error, which we do not report here, is also bounded for all three methods in this example. 
For such matrices, 
we observe in Figure~\ref{Block skew-integrals-apm} that the two first integrals of Proposition~\ref{first integral projected inner original Equations} for $k=0$ and $k=1$  are preserved for the APM, see also Remark~\ref{invariants for Arnoldi projection}.  
\begin{figure}[h]
\centering
 \begin{subfigure}[b]{0.33\textwidth}
 \centering
              \includegraphics[width=0.75\textwidth]{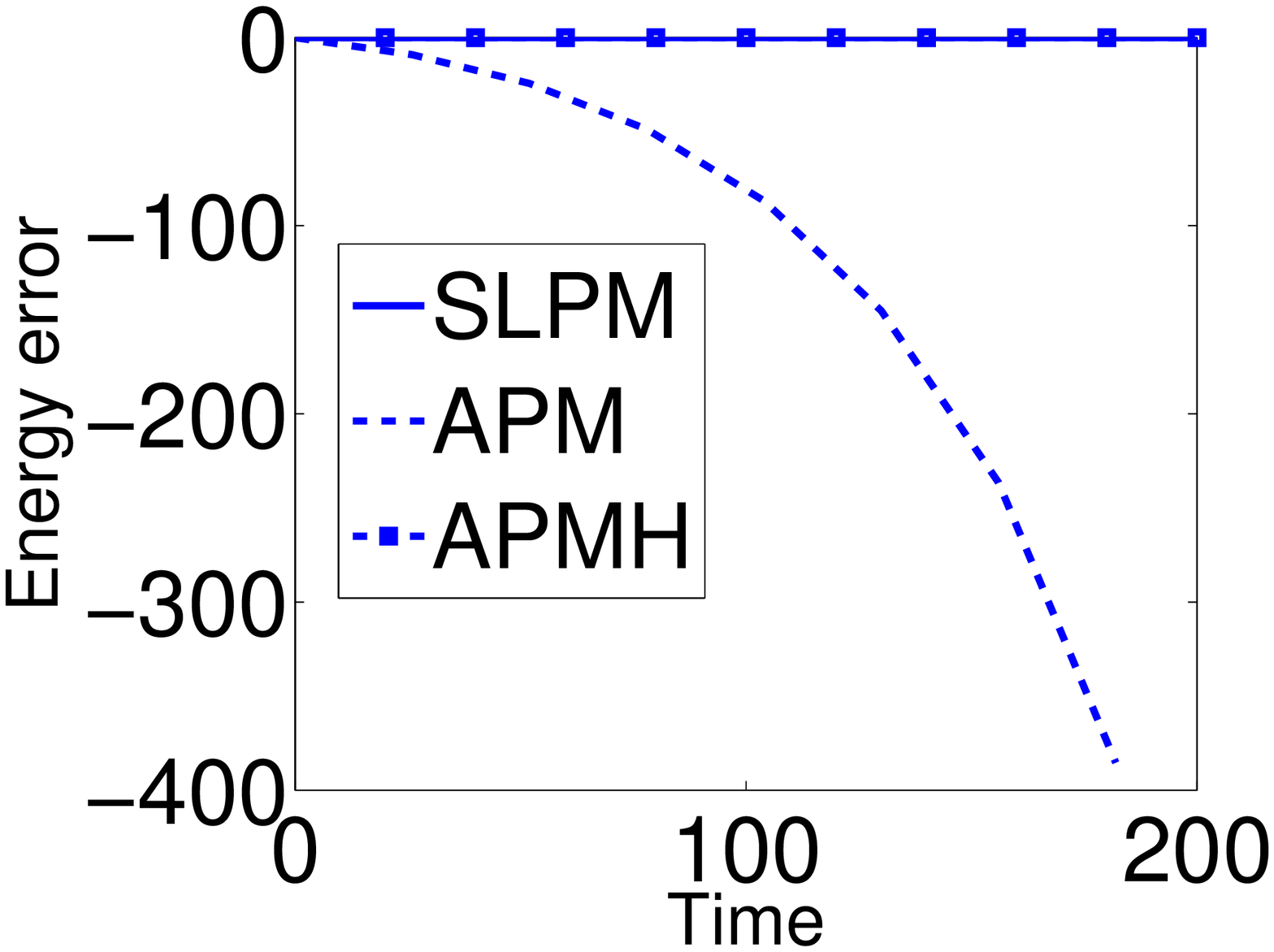}
                 \caption{Energy error}\label{Block diagonal-energy error-apm}
        \end{subfigure}
        \begin{subfigure}[b]{0.33\textwidth}
        \centering
                \includegraphics[width=0.75\textwidth]{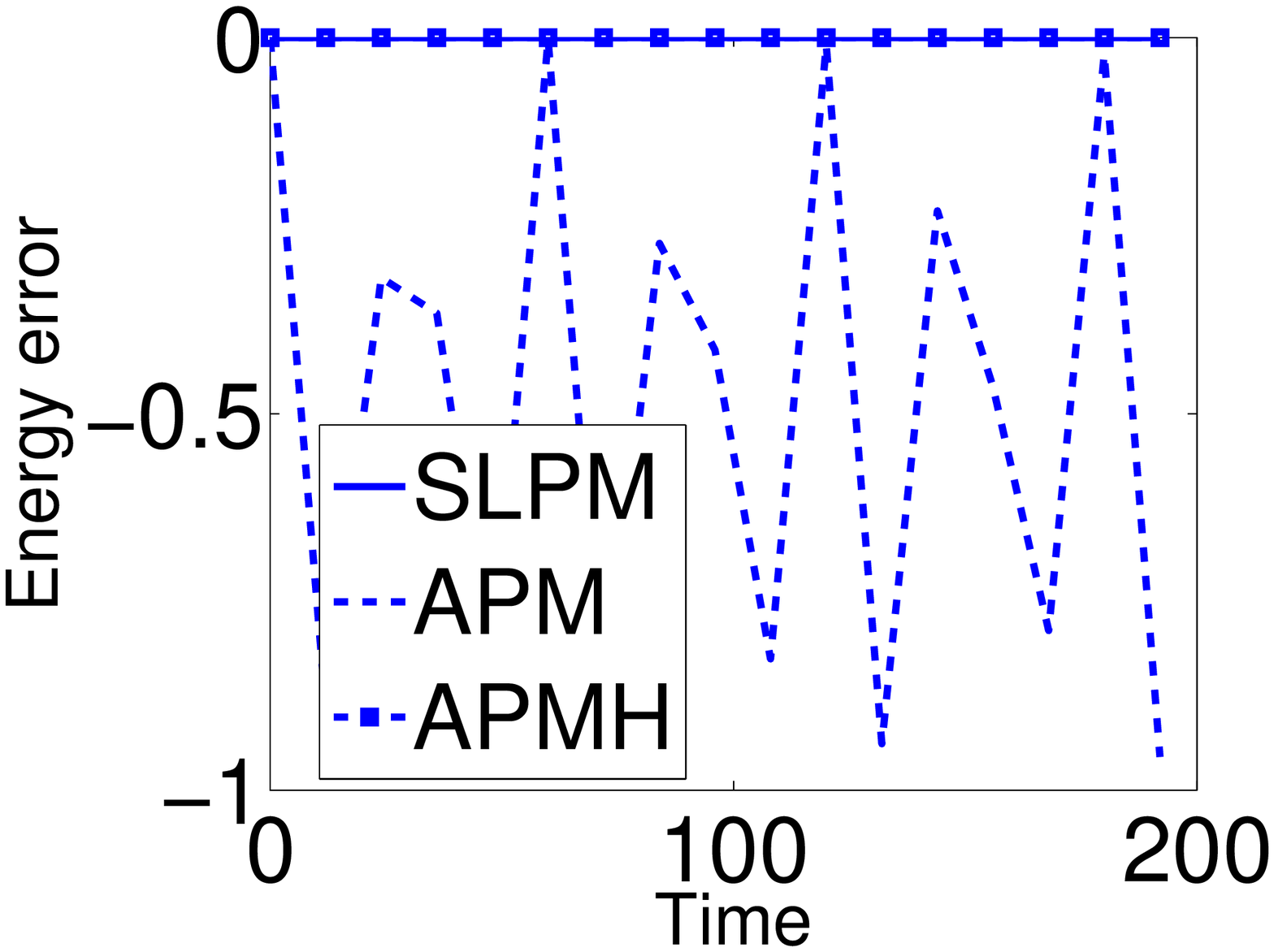}
                 \caption{\small{Energy error}}\label{Block skew-energy error-apm}
        \end{subfigure}
\begin{subfigure}[b]{0.32\textwidth}
\centering
                \includegraphics[width=0.75\textwidth]{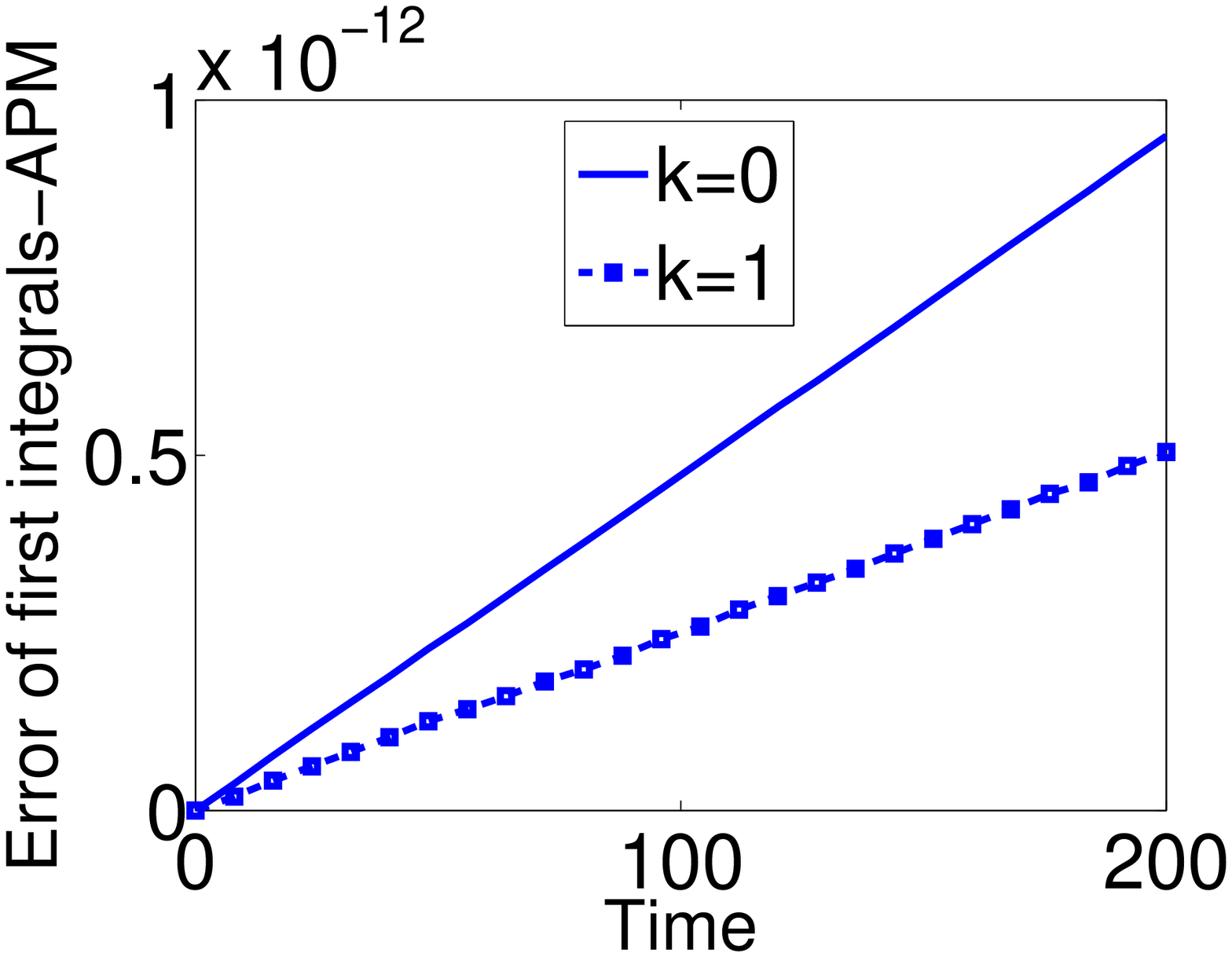}
                 \caption{First integrals, APM}\label{Block skew-integrals-apm}
        \end{subfigure}
\caption{\small{Methods without restart. In Figure \ref{Block diagonal-energy error-apm} a block diagonal Hamiltonian matrix is considered. In Figure \ref{Block skew-energy error-apm} and \ref{Block skew-integrals-apm}, we consider a skew-symmetrix, Hamiltonian matrix $A$. }}\label{verify APM-JA=AJ}
\vspace{-15pt}
\end{figure}

\subsubsection{Full matrices: Comparison of APMH, SLPM, BJPM}\label{APM versus SLPM versus model reduction}
In this subsection, we consider a randomly generated, full Hamiltonian matrix $A=JH$. 
In Figure \ref{Firstintegral-error-comp-anymatrix}, we use the methods without restart.  Figure ~\ref{firstintk01_apmh_anymatr_anyini} shows that the first integrals of  Proposition~\ref{first integral projected inner original Equations} are
 preserved by APMH on a moderately large time interval and for a general Hamiltonian matrix (i.e. imposing only that $JA$ is symmetric). However, even if the error is of size $10^{-14}$, there is a clear drift in the first integrals. A similar drift is observed also in the energy error for all methods without restart. 
In Figure \ref{Enror-error-comp-anymatrix} we use the restart technique. The energy is well preserved for all three methods, the global error grows slowly and linearly (and it remains bounded also for larger times intervals).  Here and in other experiments, the BJPM perform better in the energy error and global error compared to the other considered methods. As previously mentioned, the restart procedure destroys the preservation of first integrals of Proposition~\ref{first integral projected inner original Equations}.

In Figure~\ref{convergence}, we report convergence plots for APMH, SLPM and BJPM, showing how the global error decreases when the dimension of the Krylov subspace increases and the end time $T$ is fixed. 
We observe that the global error decreases to almost $10^{-15}$, and all the methods converge. $T$ is equal to $2$ in this experiment, but the methods converge well also for larger end time, such as $T=200$.

\begin{figure}[h]
\centering
\vspace{-15pt}
      \begin{subfigure}[b]{0.49\textwidth}
      \centering
                \includegraphics[width=0.75\textwidth]{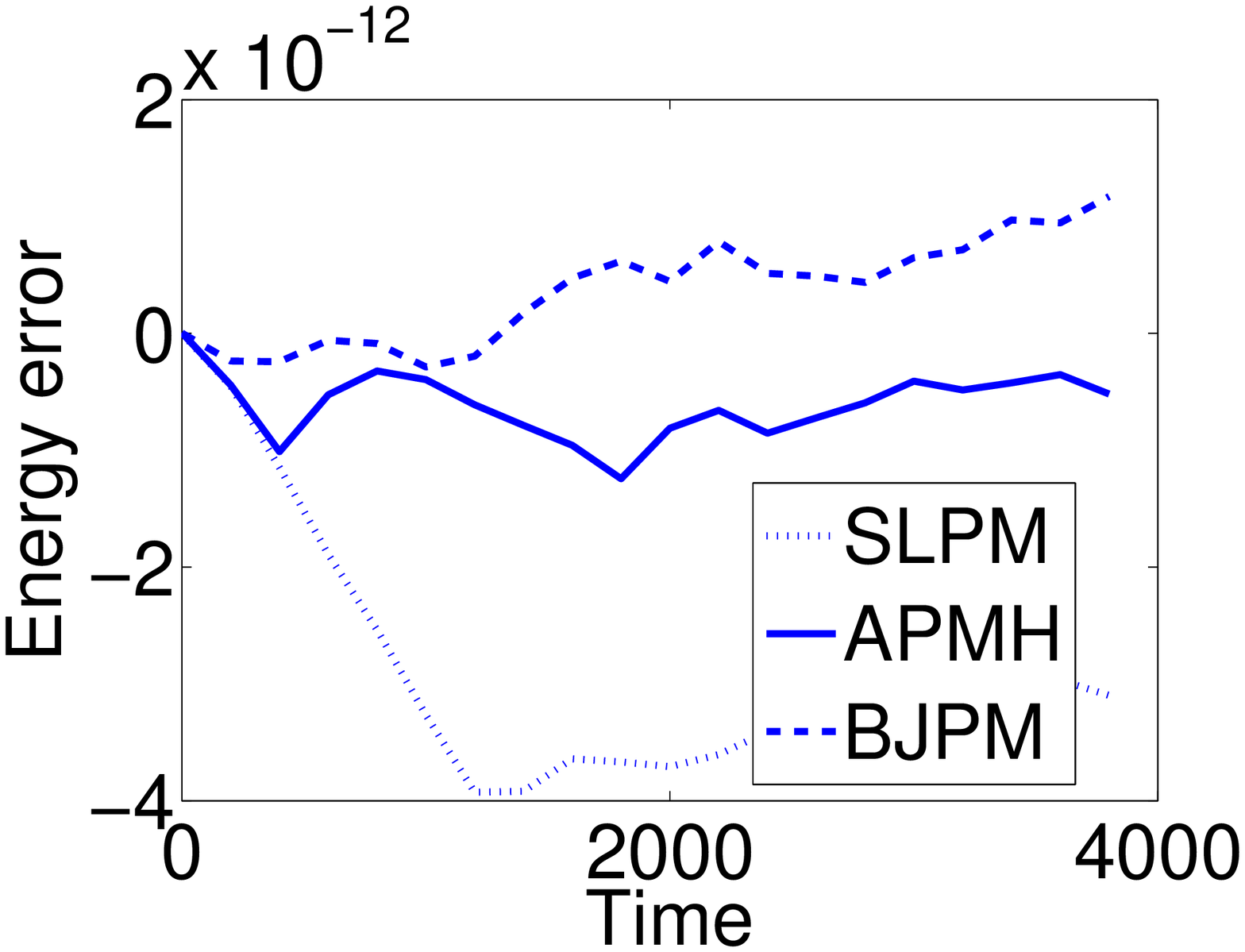}
                 \caption{Energy error}\label{ener_slpm_apmh_bjpm1_anymatr_anyini}
        \end{subfigure}
        \begin{subfigure}[b]{0.49\textwidth}
        \centering
                \includegraphics[width=0.75\textwidth]{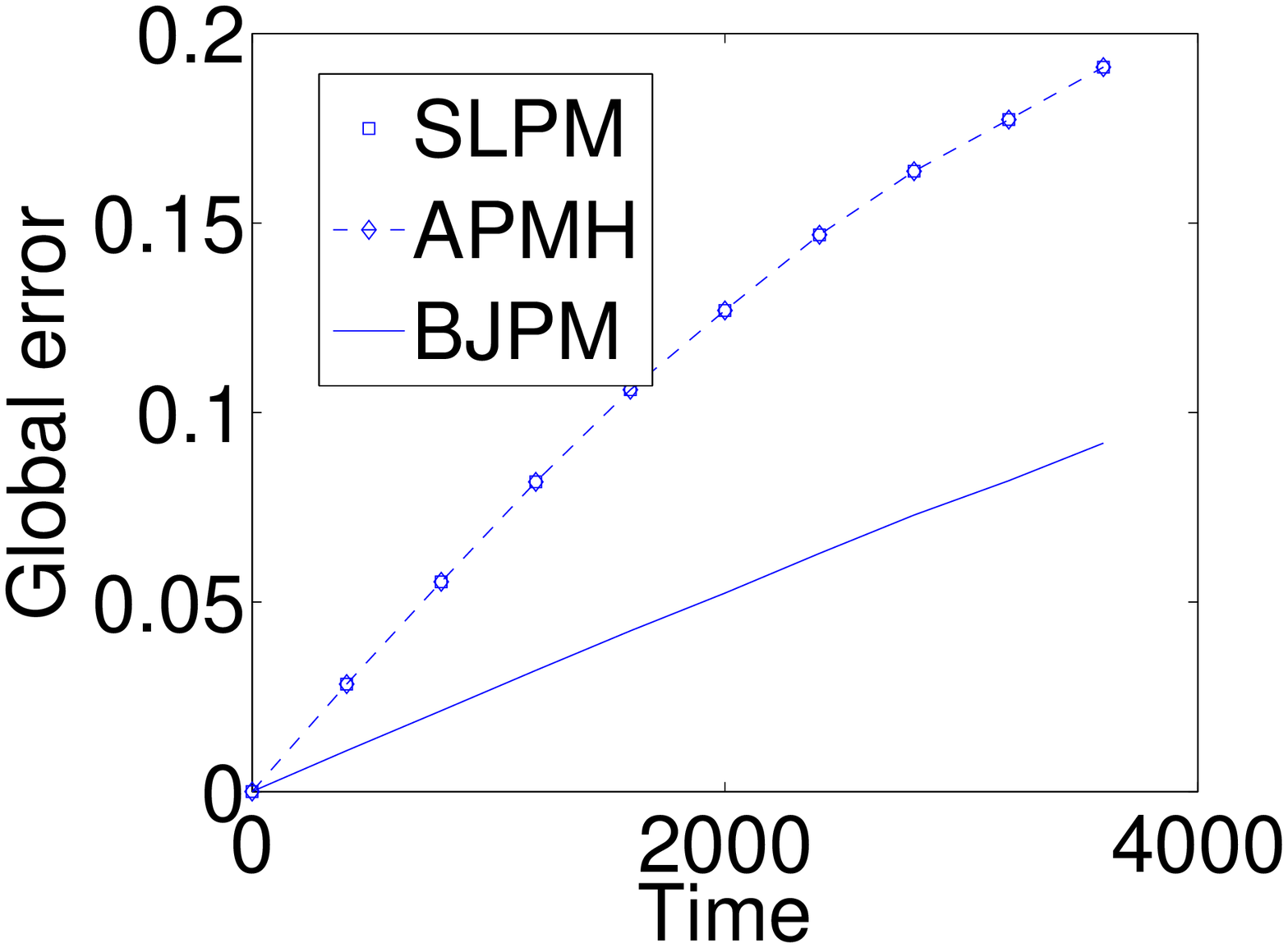}
                 \caption{Global error}\label{er_slpm_apmh_bjpm1_anymatr_anyini_restart.eps}
        \end{subfigure}
      \caption{{\small Methods with restart.}}\label{Enror-error-comp-anymatrix}
\vspace{-20pt}
\end{figure}

\begin{figure}[h]
\centering
\vspace{-20pt}
      \begin{subfigure}[b]{0.49\textwidth}
      \centering
                \includegraphics[width=0.75\textwidth]{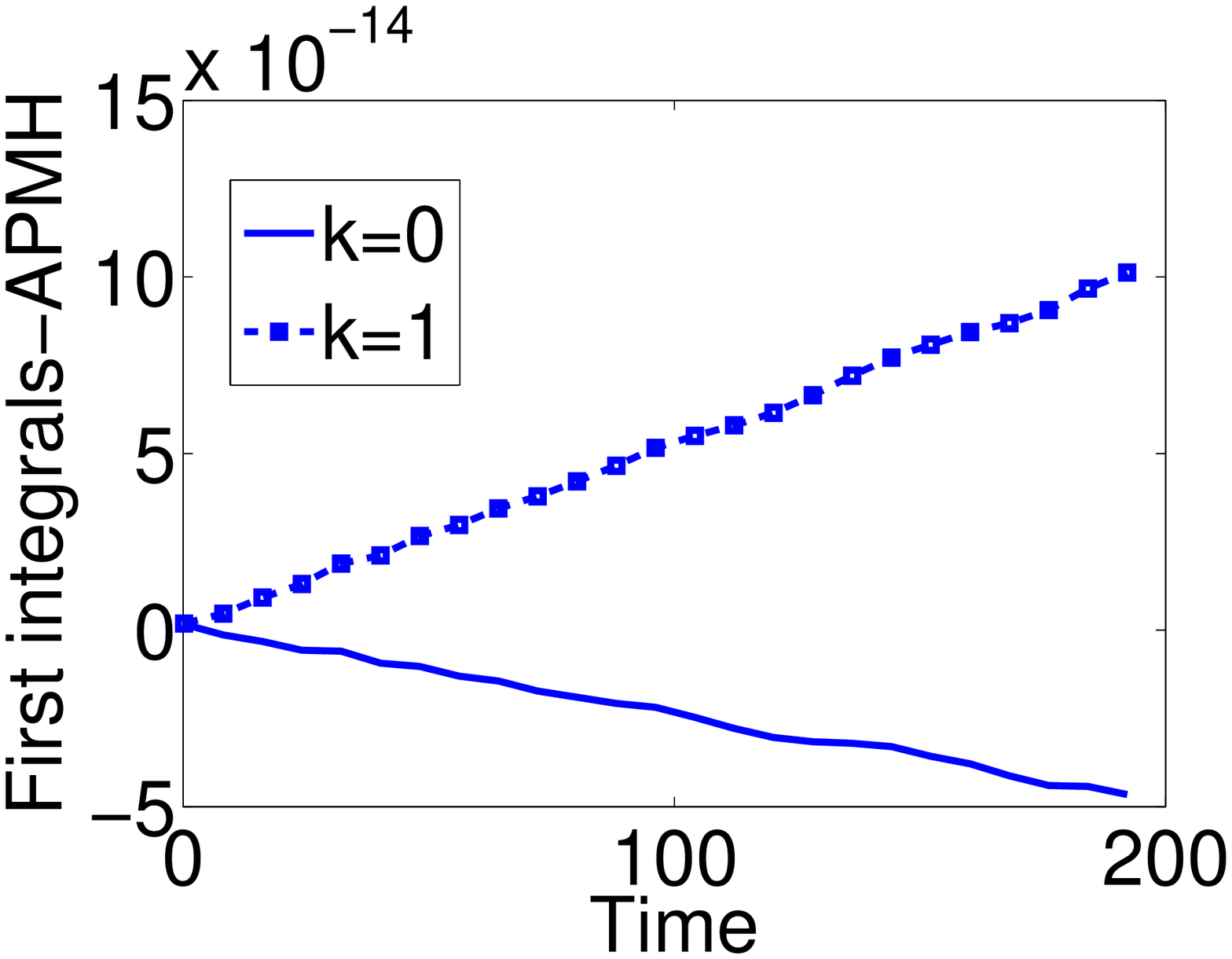}
                 \caption{\small{First integrals, APMH}}\label{firstintk01_apmh_anymatr_anyini}
        \end{subfigure}
        \begin{subfigure}[b]{0.49\textwidth}
        \centering
                \includegraphics[width=0.75\textwidth]{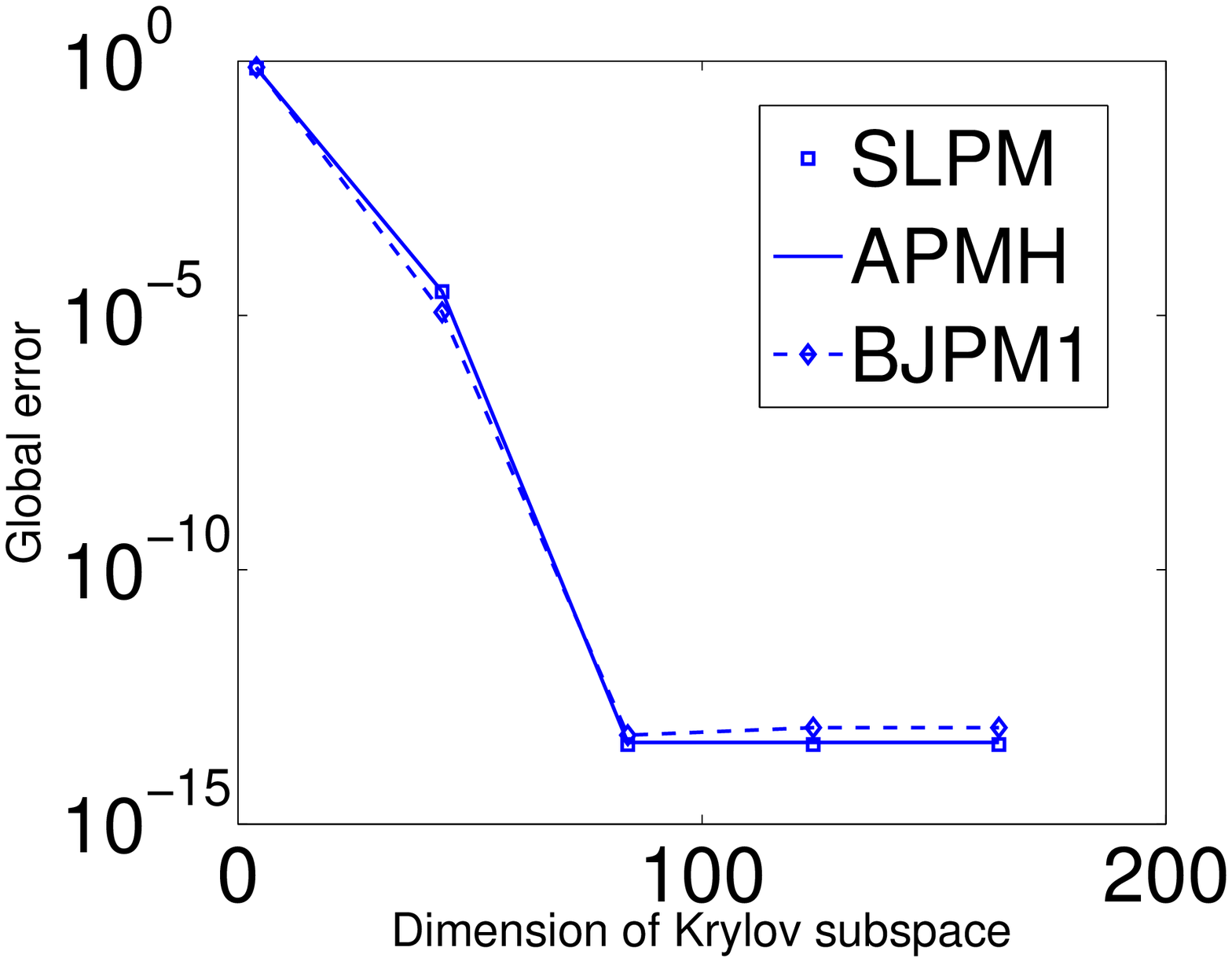}
                 \caption{Convergence}\label{convergence}
        \end{subfigure}
      \caption{{\small Methods without restart. In Figure \ref{convergence}, the global error at the end time $T=2$ is shown.}}\label{Firstintegral-error-comp-anymatrix}
\vspace{-10pt}
\end{figure}

\subsubsection{Case $H_{1,2}=O$, $H_{2,2}=I$: Model reduction}
In Figure \ref{Enror-error-Arnoldi-redup spematrix with q0}, we consider a Hamiltonian matrix $A$ of the special form \eqref{special2-y form} with an initial vector of the form $y_{0}=(0,p_{0}^{T})^{T}$. We use the Arnoldi algorithm with matrix $-H_{11}$ and vector $p_{0}$ to generate the orthogonal matrix $V_{n}$ in the model reduction procedure described in Section~\ref{specail structure of Hamiltonian matrix}. The methods behave as predicted. The APM behaves very well in this case and similarly to the methods based on model reduction, see Sec \ref{specail structure of Hamiltonian matrix}.
Energy preservation is shown in Figure \ref{Enror-Arnoldi-redup spematrix with q0} 
and bounded numerical error is shown in Figure~\ref{Error-Arnoldi-redup spematrix with q0}. 
Notice that we cannot apply the restart technique in this case because the special form of the initial vector will in general not be maintained from the first to the second subinterval. 

\begin{figure}[h]
\centering
\vspace{-10pt}
      \begin{subfigure}[b]{0.49\textwidth}
      \centering
                \includegraphics[width=0.75\textwidth]{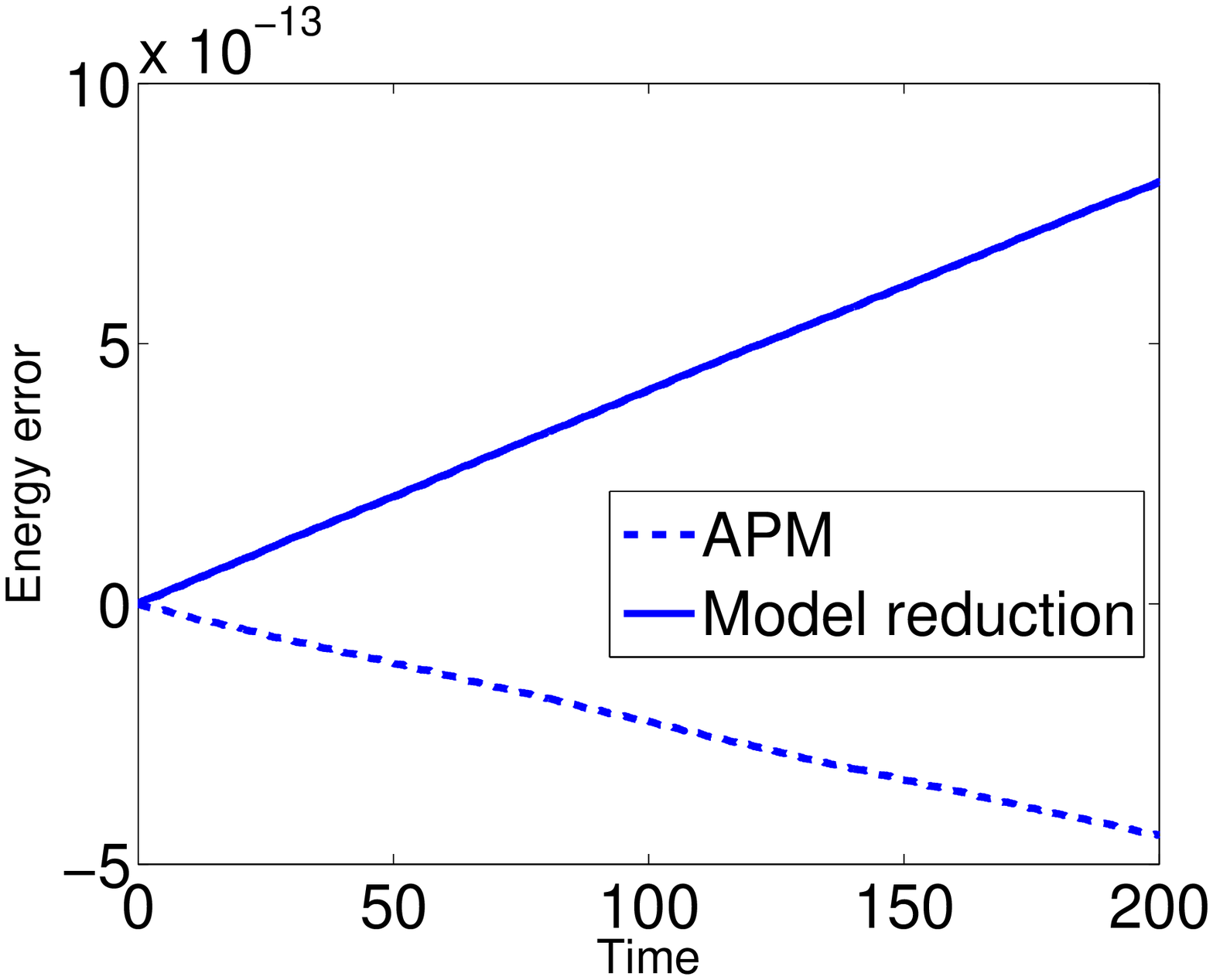}
                 \caption{Energy error}\label{Enror-Arnoldi-redup spematrix with q0}
        \end{subfigure}
        \begin{subfigure}[b]{0.49\textwidth}
        \centering
                \includegraphics[width=0.75\textwidth]{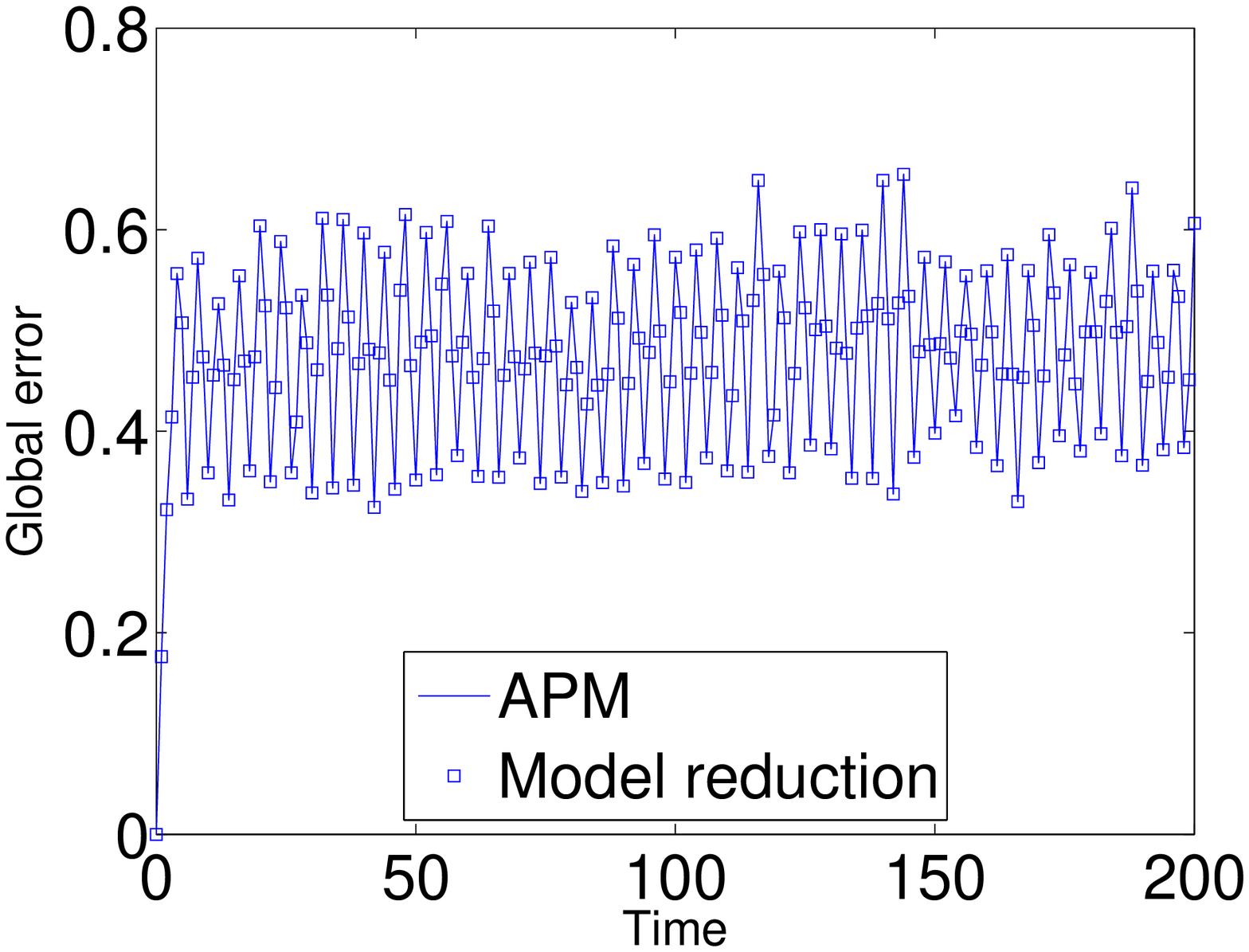}
                 \caption{Global error}\label{Error-Arnoldi-redup spematrix with q0}
        \end{subfigure}
      \caption{{\small APM compared to a procedure of model reduction. }}\label{Enror-error-Arnoldi-redup spematrix with q0}
\vspace{-30pt}
\end{figure}

\subsection{Hamiltonian PDEs}
In this section we apply the methods to the wave equations and  the \text{Maxwell's Equations}.
\vspace{-15pt}
\subsubsection{Wave equation}
We consider the $2\mathrm{D}$ wave equations
\begin{equation}\label{rewrite-2dwave Equations}
\dot{\phi}=\psi,\qquad
\dot{\psi}=\triangle \phi,
\end{equation}
on $[0,1]\times [0,1]$ with homogeneous Dirichlet boundary conditions  
$\phi(t,0,y)=\phi(t,1,y)=\phi(t,x,0)=\phi(t,x,1)=0$
and a randomly generated initial vector. 
Semi-discretizing on an equispaced  grid $x_i=i\, \Delta x$ and $y_j=j\, \Delta y$, $\Delta x=\Delta y$, $i,j=0,\dots, N$ and assuming $u(x_i,y_j)\approx U_{i,j}$, we obtain a system 
\begin{equation}\label{semidiscretete-2dwave Equations}
\dot{U}=AU,\quad
U(0)=U_{0},\qquad
A = \left[ \begin{matrix}
0 & I\\
G &  0 
\end{matrix} \right]
\end{equation}
\noindent with $G$ the discrete 2D Laplacian obtained by using 
central differences.
This is a Hamiltonian system with energy 
$\mathcal{H}=\frac{1}{2}U^{T}JAU\equiv\frac{1}{2}U(0)^{T}JAU(0).$
We perform experiments with all the Krylov projection methods discussed in this paper. Figure \ref{ener_max1d} shows that all the methods are energy-preserving. Figure \ref{firstintk01_apmh_wave2d}, shows that first integrals are preserved by APMH.

\begin{figure}[h]
\centering
 \begin{subfigure}[b]{0.32\textwidth}
 \centering
                \includegraphics[width=0.75\textwidth]{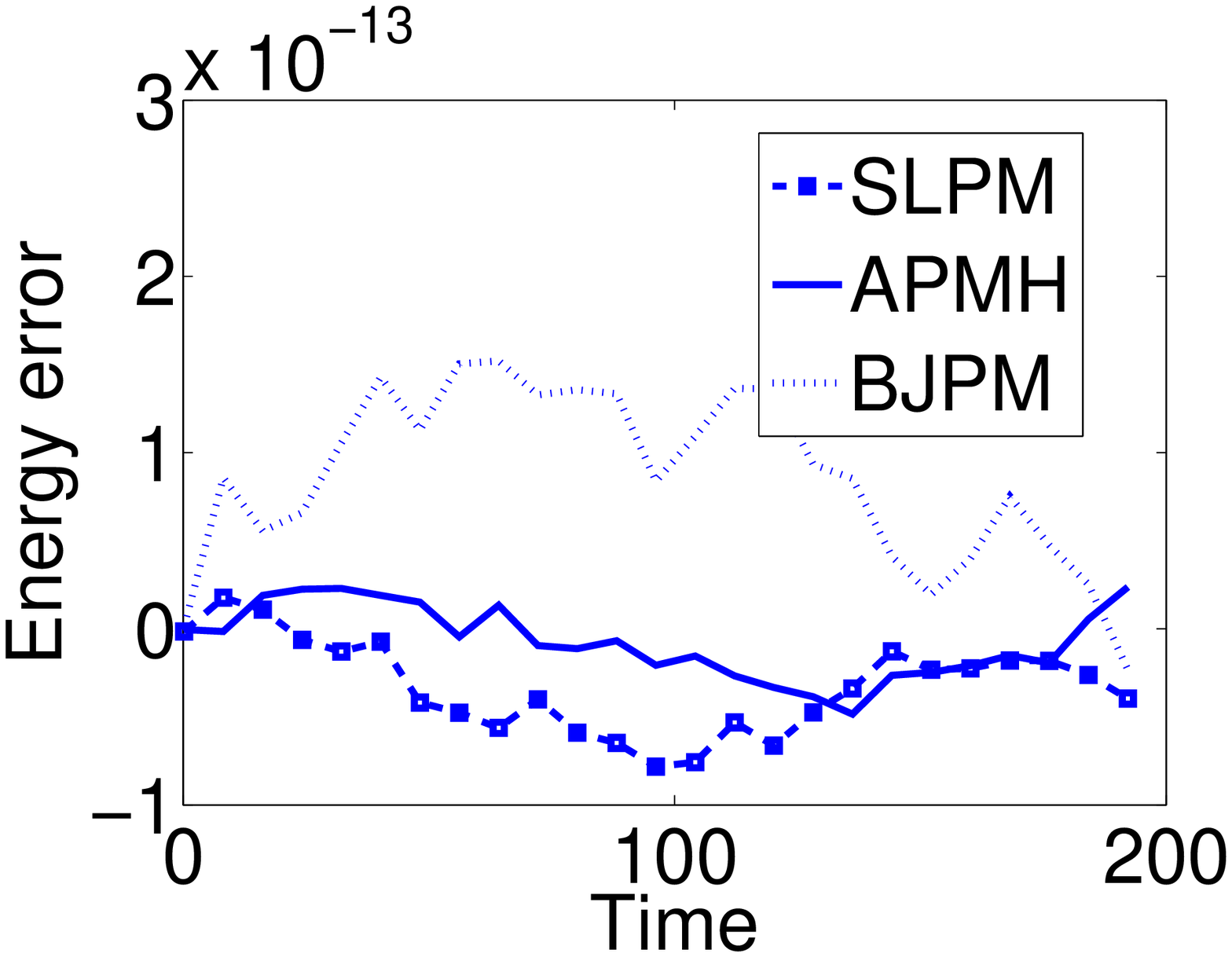}
                 \caption{Energy error}\label{ener_max1d}
        \end{subfigure}
        \begin{subfigure}[b]{0.32\textwidth}
        \centering
                \includegraphics[width=0.75\textwidth]{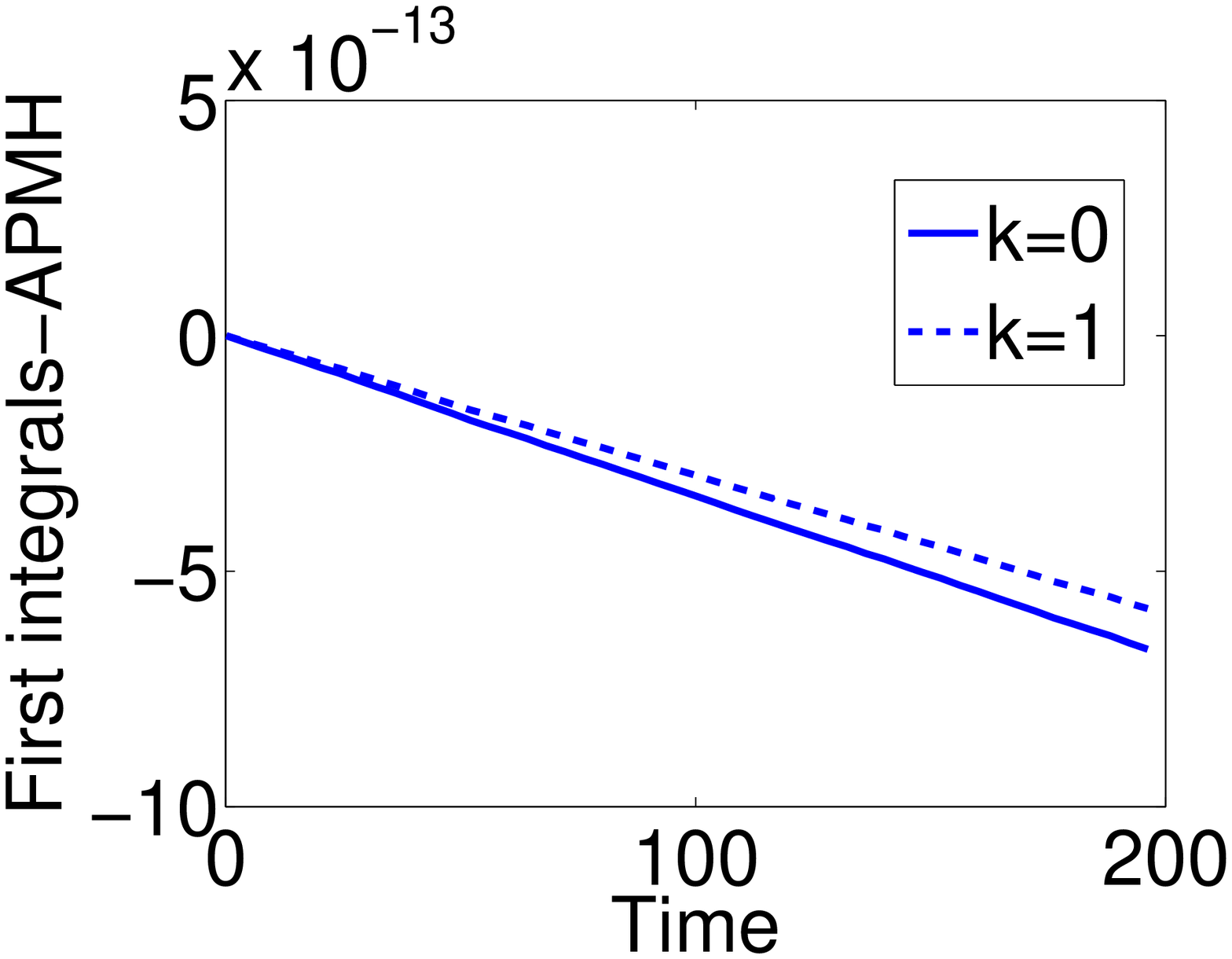}
                 \caption{{\small First integrals-APMH}}\label{firstintk01_apmh_wave2d}
        \end{subfigure}
\begin{subfigure}[b]{0.32\textwidth}
\centering
                \includegraphics[width=0.75\textwidth]{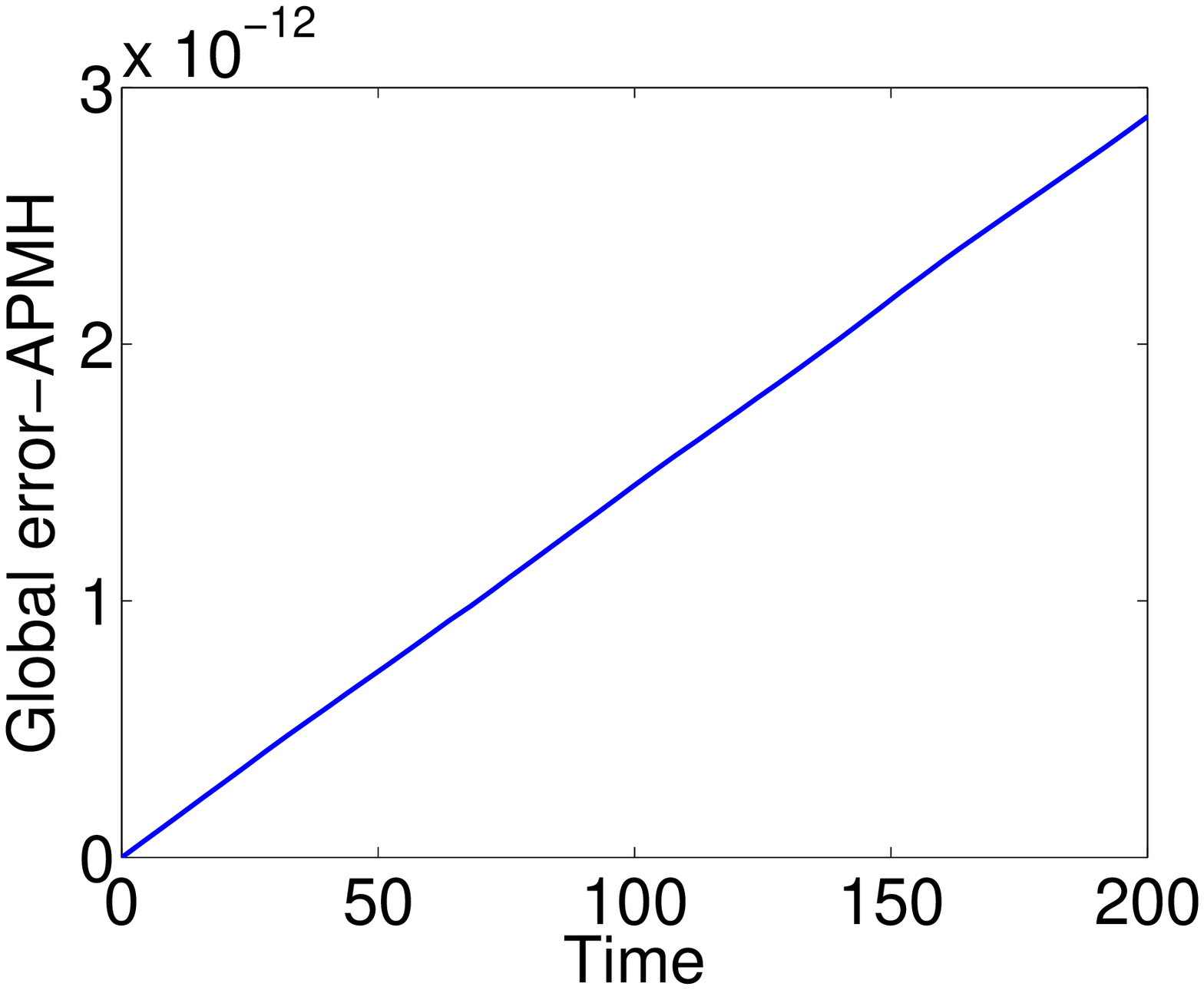}
                 \caption{{\small Global error-APMH}}\label{Errors_maxwell1d_inner}
        \end{subfigure}
\caption{\small{Wave equation in $2d$. Figure \ref{ener_max1d} energy error, methods with restart. Figure \ref{firstintk01_apmh_wave2d},  first integrals, methods without restart. \text{Maxwell's} equations in $1$d, Figure \ref{Errors_maxwell1d_inner}, global error, methods without restart.
}}
\vspace{-15pt}
\end{figure}

\subsubsection{$1\mathrm{D}$ \text{Maxwell's equations}}\hspace*{\fill} \vspace{-10pt}\\
We consider $1\mathrm{D}$ \text{Maxwell's equations}  
\begin{equation}\label{1dMaxwell Equations}
\begin{split}
\partial_{t} E&=\partial_{x}B,\\
\partial_{t} B&=\partial_{x}E
\end{split}
\end{equation}
for $x\in [0,1]$ and $t>0$ with boundary conditions $E(0,t)=E(1,t)=0, \, B_{x}(0,t)=B_{x}(1,t)=0$ and initial conditions $E(x,0)=\sin(\pi x)$ and $B(x,0)=\cos(\pi x)$.
After \text{semi-discretization} with $E(x_i,t)\approx E_i(t)$ and $B(x_i,t)\approx B_i(t)$, $i=0,\dots , N$, we get a system of ODEs 
\begin{equation}\label{semidiscrete 1dMaxwell Equations}
\dot{U}=\bar{S}DU, \quad U(0)=U_0,
\end{equation}
where $U=[E_{1},...,E_{N-1},B_{0},...,B_{N}]^{T}$ and 
\begin{equation*} 
{\small \bar{S}=\frac{1}{2h} \left[
\begin{matrix}
0_{N-1,N+1} & G\\
\vspace{4mm}
\\
 -G^{T} & 0_{N+1,N-1}
\end{matrix}\right]}, \qquad
{\small G = \left[ \begin{matrix} 
-2 & 0 & 1 & & \\
\vspace{1mm}
    & -1 & 0 & 1 &  \\
\vspace{1mm}
& & \ddots & \ddots & \ddots & \\
\vspace{1mm}
& & & -1 & 0 & 1\\
\vspace{1mm}
& & & & -1 & 0 & &2 \end{matrix} \right]}
\end{equation*}
and $D=\mathrm{diag}(I_{N-1},\frac{1}{2},I_{N-1},\frac{1}{2})$.
Equation \eqref{semidiscrete 1dMaxwell Equations} fits the framework of 
section \ref{sec:APM}, with $\bar{S}$ skew-symmetric and $D$ symmetric and positive definite, therefore APMH can be applied to this problem. The numerical approximation of  $U$ obtained applying the APMH 
preserves the first integrals $\mathcal{H}_{k}(\bar{U})$ of
Proposition \ref{first integral projected inner original Equations}. The first integrals are preserved with an error of about  $10^{-13}$ (not reported here). 
In Figure \ref{Errors_maxwell1d_inner}, we show the global error and observe that the problem is solved with high accuracy.
\vspace{-10pt}

\subsection{Numerical results for $3\mathrm{D}$ \text{Maxwell's equations}}
We consider $3\mathrm{D}$ \text{Maxwell's equations} in CGS units for the electromagnetic field in a vacuum
\begin{equation}\label{3dMaxwell Equations}
\begin{split}
\partial_{t} E&=-c\nabla\times B,\\
\partial_{t} B&=c\nabla\times E.
\end{split}
\end{equation}
The boundary conditions are zero and the initial conditions are randomly generated for both fields. We consider $c=1$. We get the following Hamiltonian system after \text{semi-discretization}:
\vspace{-20pt}
\begin{multicols}{2}
\vspace{-6pt}
\begin{equation*}\label{semidiscrete 3dMaxwell Equations}
\dot{U}=AU, \qquad U(0)=U_{0},
\end{equation*}\quad
\begin{equation}\label{semidiscrete 3dMaxwell Equations}
A= \left[
\begin{matrix}
0 & -G_{1}\\
 G_{1} & 0
\end{matrix}\right],
\end{equation}
\end{multicols}
\noindent where $U=[E_{1,1,1},...,E_{N-1,N-1,N-1},B_{1,1,1},...,B_{N-1,N-1,N-1}]^{T}$ and $G_{1}$, symmetric and of the size $(N-1)^3$, is the discretization of the curl operator $\nabla\times$. 
\vskip-0.3cm
\begin{remark}\label{first integral projected semidiscrete 3dMaxwell Equations-apmh}
The matrix $A$ is skew-symmetric in equation \eqref{semidiscrete 3dMaxwell Equations}. Therefore 
the APMH with $J=A, H=I$ applied to the system \eqref{semidiscrete 3dMaxwell Equations},  equals the APM and preserves the first integrals $\mathcal{H}_{k}(\bar{U})$ 
of Proposition \ref{first integral projected inner original Equations}. 
\end{remark}
\begin{remark}\label{first integral projected semidiscrete 3dMaxwell Equations-Hamiltonian}
Equation \eqref{semidiscrete 3dMaxwell Equations} can be rewritten as a Hamiltonian equation
$\dot{U}=JHU,$ with $H=J^{-1}A$ a symmetric matrix. Therefore we can also apply SLPM and BJPM to system \eqref{3dMaxwell Equations} and the energy $\mathcal{H}(U)= \frac{1}{2}U^{T}J^{-1}AU$ is preserved. However, APMH cannot be used here because $H$ is not a positive definite matrix, and the inner product $\langle \cdot, \cdot\rangle_H$ is degenerate. This can lead to instabilities and both  global error and energy error might blow up during the iteration.
\end{remark}
\vspace{-5pt}
\vspace{-15pt}
\begin{figure}[h]
\centering
 \begin{subfigure}[b]{0.32\textwidth}
 \centering
                \includegraphics[width=0.75\textwidth]{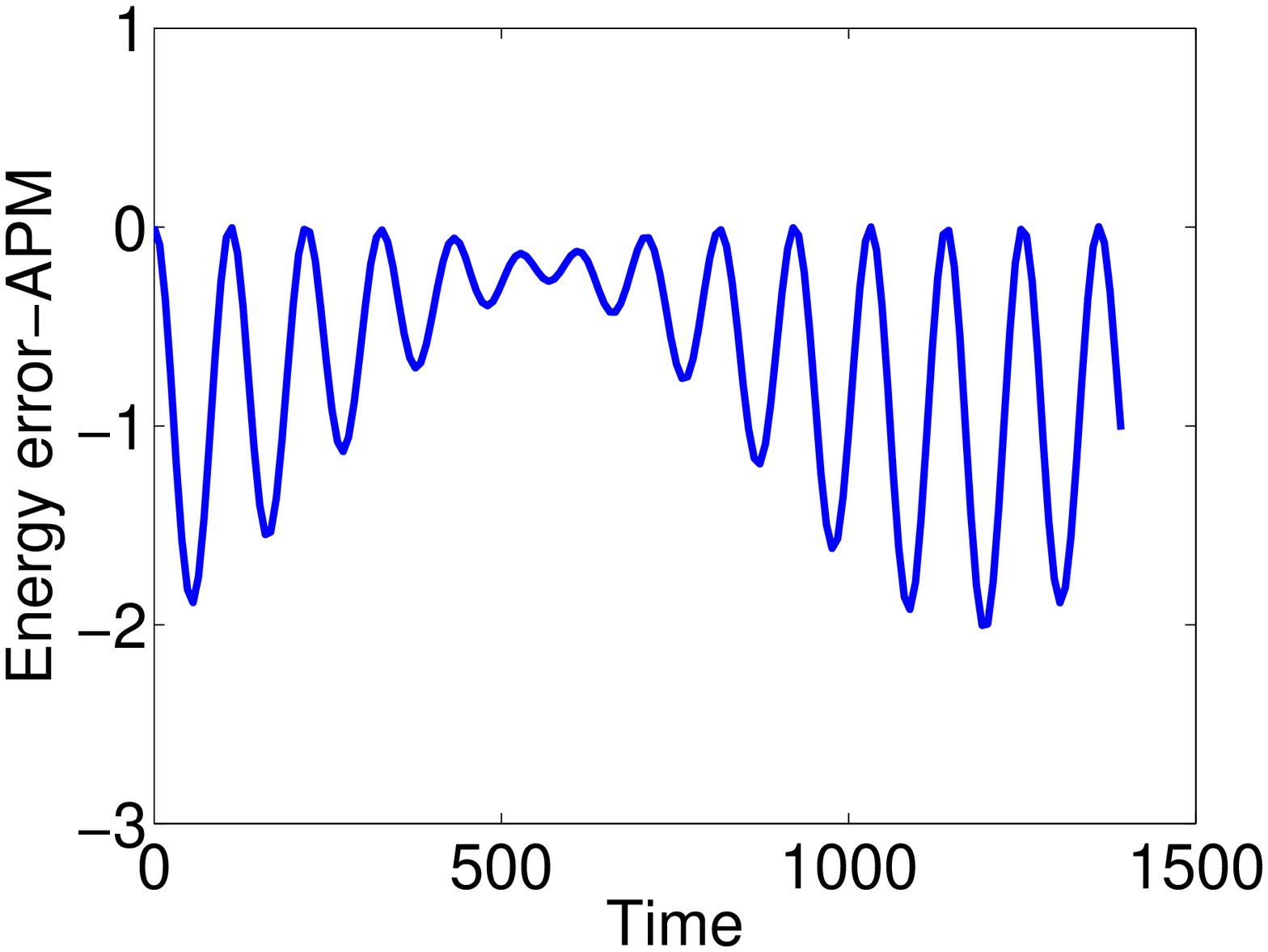}
                 \caption{Energy error}\label{ener_apm_apmh_3dmax}
        \end{subfigure}
        \begin{subfigure}[b]{0.32\textwidth}
        \centering
                \includegraphics[width=0.75\textwidth]{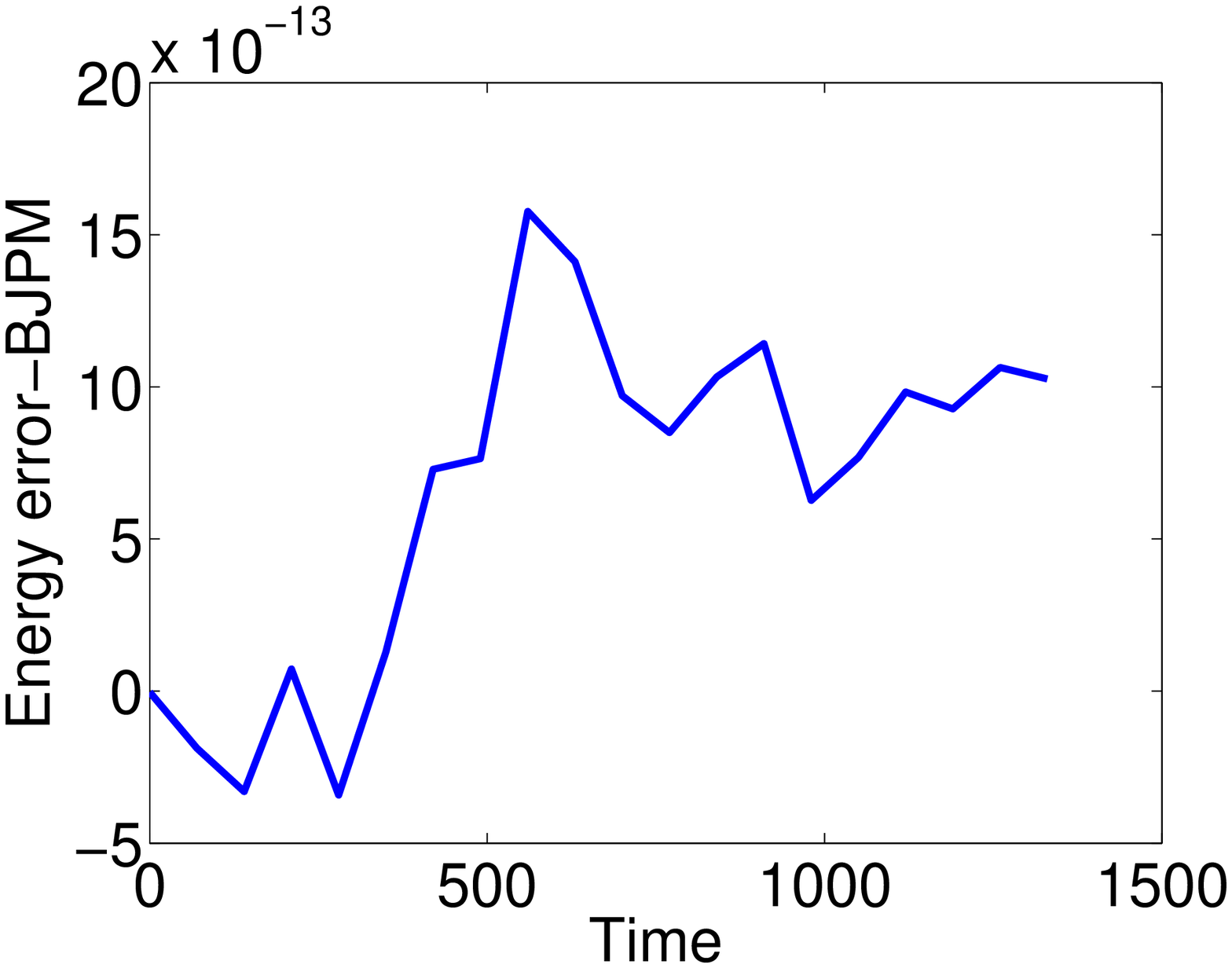}
                 \caption{\small{Energy error, BJPM}}\label{ener_slpm_bjpm1_3dmaxwell}
        \end{subfigure}
\begin{subfigure}[b]{0.32\textwidth}
\centering
                \includegraphics[width=0.75\textwidth]{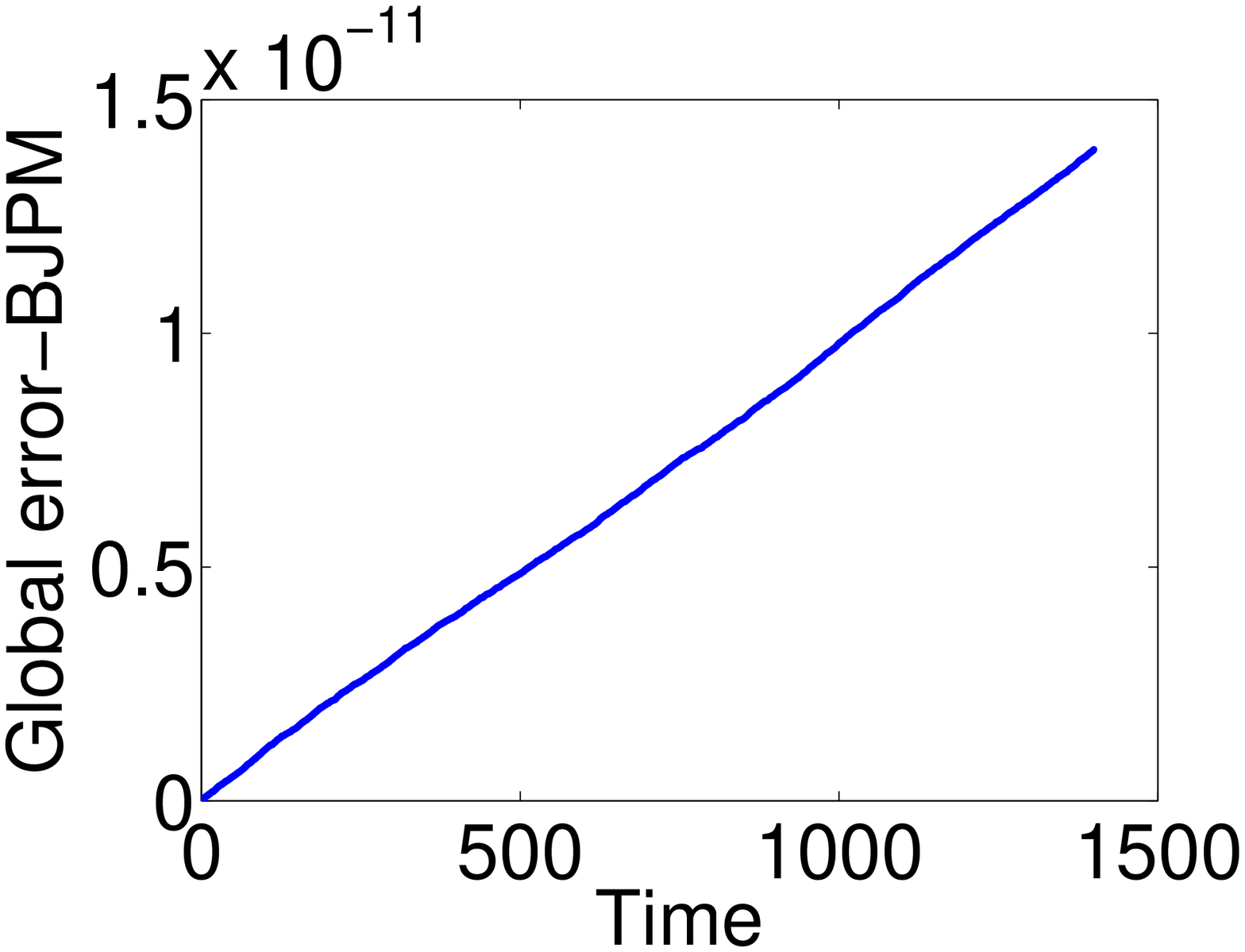}
                 \caption{\small{Global error, BJPM}}\label{error_bjpm_7000steps_3dmax}
        \end{subfigure}
\caption{{\small The dimension of Krylov subspace is set to be $4$ in Figure \ref{ener_apm_apmh_3dmax} and $16$ in Figure \ref{ener_slpm_bjpm1_3dmaxwell} and \ref{error_bjpm_7000steps_3dmax}. In Figure \ref{ener_slpm_bjpm1_3dmaxwell} and \ref{error_bjpm_7000steps_3dmax}, the methods with the restart technique are used. Figure \ref{ener_apm_apmh_3dmax} corresponds to the energy error considered as in Remark \ref{first integral projected semidiscrete 3dMaxwell Equations-apmh}, while figure \ref{ener_slpm_bjpm1_3dmaxwell} to the energy error considered in Remark \ref{first integral projected semidiscrete 3dMaxwell Equations-Hamiltonian}.}\label{ener_apm_apmh_slpm_bjpm1_3dmaxwell}}
\vspace{-25pt}
\end{figure}

\begin{figure}[h]
\centering
 \begin{subfigure}[b]{0.49\textwidth}
 \centering
                \includegraphics[width=0.75\textwidth]{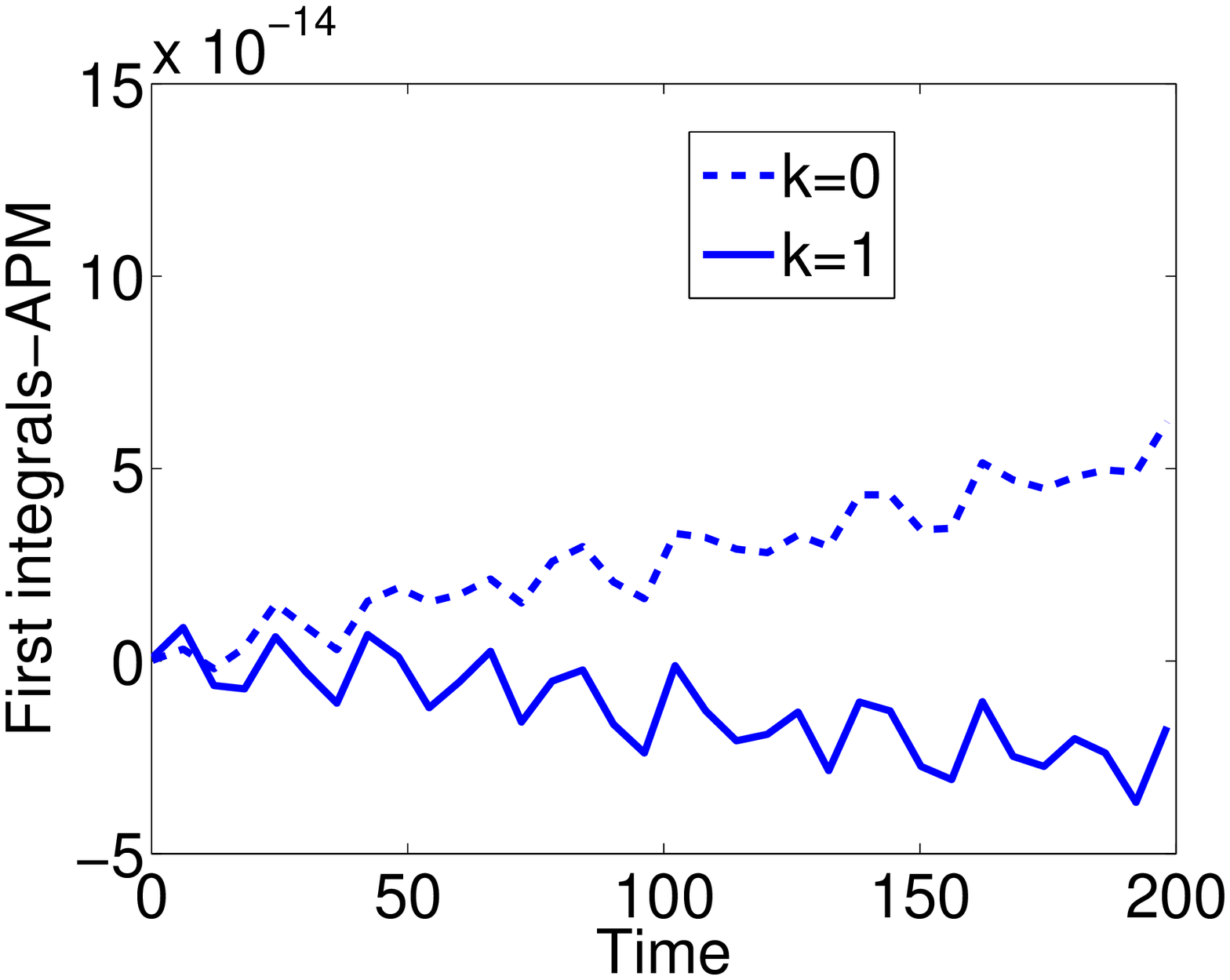}
                 \caption{{\small First integrals-APM}}\label{firintek01_apm_apmh_3dmax}
        \end{subfigure}
        \begin{subfigure}[b]{0.49\textwidth}
        \centering
                \includegraphics[width=0.75\textwidth]{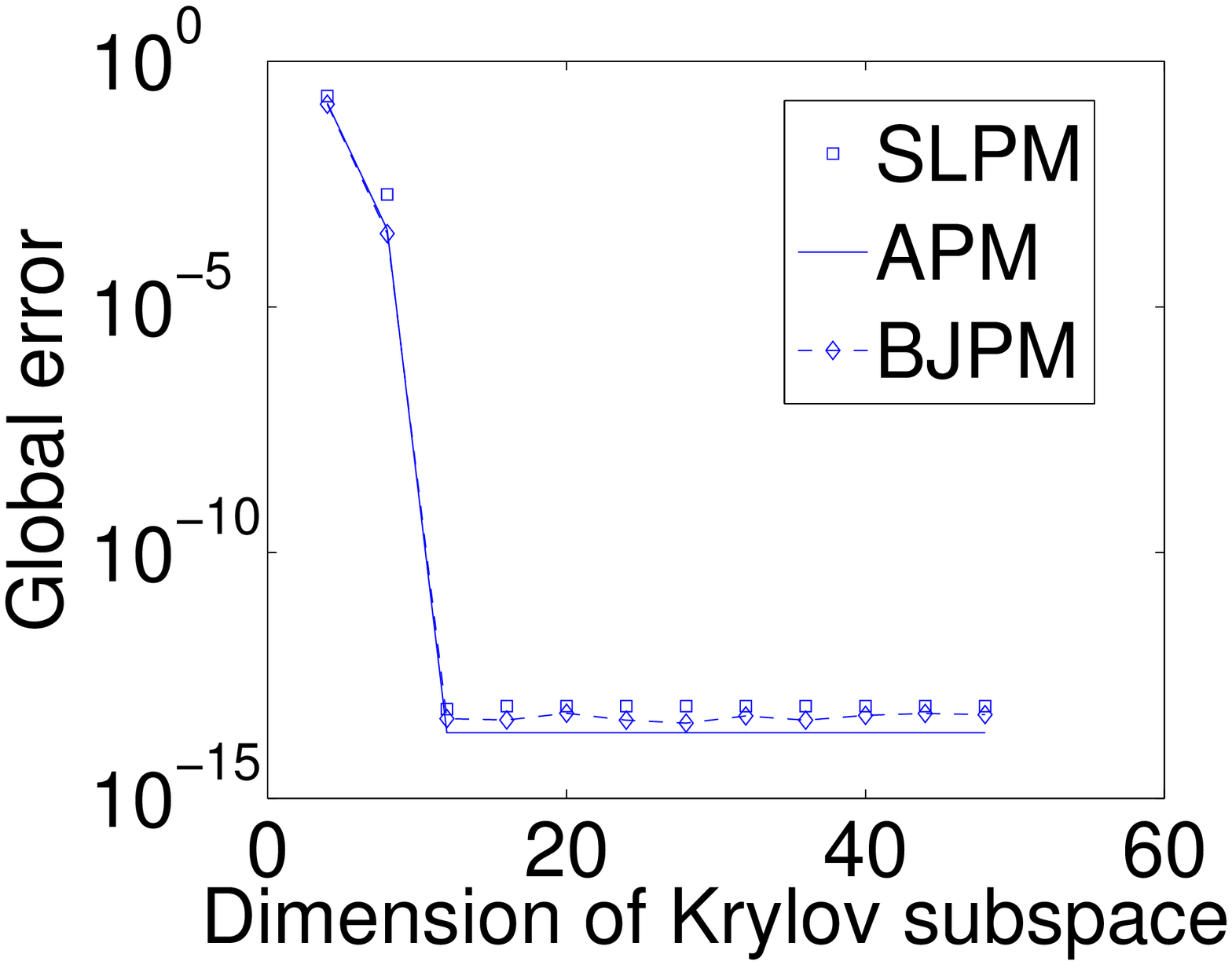}
                 \caption{Convergence}\label{conv_slpm_bjpm1_apm_apmh_3dmaxwell}
        \end{subfigure}
\caption{{\small In Figure~\ref{firintek01_apm_apmh_3dmax}, the dimension of Krylov subspace is set to be $4$. In Figure~\ref{conv_slpm_bjpm1_apm_apmh_3dmaxwell} we consider $L^{2}$ norm of the global error at $t=T=2$ as a function of the dimension of the Krylov subspace.}}\label{er_ener_3dmaxwell}
\end{figure}

Figure \ref{ener_apm_apmh_3dmax} shows that the energy error of APM is bounded as stated in Remark \ref{first integral projected semidiscrete 3dMaxwell Equations-apmh}. The energy error of APM will decrease to $10^{-12}$ when we increase the dimension of the Krylov subspace from $4$ to $16$ (not shown here). Figure~\ref{ener_slpm_bjpm1_3dmaxwell} shows that the energy $\mathcal{H}(U)= \frac{1}{2}U^{T}J^{-1}AU$ is preserved for BJPM as stated in Remark \ref{first integral projected semidiscrete 3dMaxwell Equations-Hamiltonian}.  The problem is solved to high accuracy by BJPM with dimension of Krylov space $16$. 
As shown in Figure~\ref{firintek01_apm_apmh_3dmax}, APM preserves the first integrals in Remark \ref{first integral projected semidiscrete 3dMaxwell Equations-apmh}. 
In Figure \ref{conv_slpm_bjpm1_apm_apmh_3dmaxwell}, we report convergence plots for the methods. As the dimension of the Krylov subspace increases, the  global error decreases very fast for all the methods.

\section*{Acknowledgment} 
\vskip-0.3cm
This work was supported by the European Union's Horizon 2020 research and innovation programme under the Marie Sklodowska-Curie, grant agreement No.
691070. 
The second author would like to thank Dr. Long Pei for helpful discussions and suggestions on previous versions of this paper.
\vskip-0.3cm


\begin{thebibliography}{10}
\providecommand{\url}[1]{{#1}}
\providecommand{\urlprefix}{URL }
\expandafter\ifx\csname urlstyle\endcsname\relax
  \providecommand{\doi}[1]{DOI \discretionary{}{}{}#1}\else
  \providecommand{\doi}{DOI \discretionary{}{}{}\begingroup
  \urlstyle{rm}\Url}\fi

\bibitem{van2014port}
A.~van~der Schaft, D.~Jeltsema, Port-{H}amiltonian systems theory: {A}n
  introductory overview, Foundations and Trends in Systems and Control
  \textbf{1}(2-3), 173 (2014)

\bibitem{feng1987symplectic}
K.~Feng, M.z. Qin, in \emph{Numerical methods for partial differential
  equations} (Springer, 1987), pp. 1--37

\bibitem{mclachlan1993symplectic}
R.~McLachlan, Symplectic integration of {H}amiltonian wave equations,
  Numerische Mathematik \textbf{66}(1), 465 (1993)

\bibitem{marsden1982hamiltonian}
J.E. Marsden, A.~Weinstein, The {H}amiltonian structure of the
  {M}axwell-{V}lasov equations, Physica D: nonlinear phenomena \textbf{4}(3),
  394 (1982)

\bibitem{sun10sam}
Y.~Sun, P.~Tse, Symplectic and multisymplectic methods for {M}axwell's
  equations, J. Comp. Phys. \textbf{230}(5), 2076 (2010).
\newblock \doi{10.1016/j.jcp.2010.12.006}

\bibitem{taylor115partial}
E.~Taylor~Michael, Partial differential equations i. {B}asic theory, Applied
  Mathematical Sciences \textbf{115}

\bibitem{richtmyer1967difference}
R.D. Richtmyer, K.W. Morton, Difference methods for initial-value problems
  (1967)

\bibitem{labudde1975energy}
R.A. LaBudde, D.~Greenspan, Energy and momentum conserving methods of arbitrary
  order for the numerical integration of equations of motion, Numerische
  Mathematik \textbf{25}(4), 323 (1975)

\bibitem{mclachlan1999geometric}
R.I. McLachlan, G.~Quispel, N.~Robidoux, Geometric integration using discrete
  gradients, Philosophical Transactions of the Royal Society of London A:
  Mathematical, Physical and Engineering Sciences \textbf{357}(1754), 1021
  (1999)

\bibitem{brugnano10hbv}
L.~Brugnano, F.~Iavernaro, D.~Trigiante, Hamiltonian boundary value methods
  (energy preserving discrete line integral methods), J. Numer. Anal. Ind.
  Appl. Math \textbf{5}(1), 17 (2010)

\bibitem{celledoni2012preserving}
E.~Celledoni, V.~Grimm, R.I. McLachlan, D.~McLaren, B.~Owren, D.~O{'}Neale,
  G.~Quispel, Preserving energy resp. dissipation in numerical {PDE}s using the
  {"}{A}verage {V}ector {F}ield{"} method, Journal of Computational Physics
  \textbf{231}(20), 6770 (2012)

\bibitem{ge88lph}
Z.~Ge, J.~Marsden, Lie-{P}oisson {H}amilton-{J}acobi theory and {L}ie-{P}oisson
  integrators, Phys. Lett. A \textbf{133}, 134?139 (1988)

\bibitem{quispel08anc}
G.~Quispel, D.~McLaren, A new class of energy-preserving numerical integration
  methods, J. of Phys. A: Math. and Theor. \textbf{41}(4), 045206, 7 (2008).
\newblock \doi{10.1088/1751-8113/41/4/045206}

\bibitem{botchev2009numerical}
M.A. Botchev, J.G. Verwer, Numerical integration of damped {M}axwell equations,
  SIAM Journal on Scientific Computing \textbf{31}(2), 1322 (2009)

\bibitem{lopez2006preserving}
L.~Lopez, V.~Simoncini, Preserving geometric properties of the exponential
  matrix by block {K}rylov subspace methods, BIT Numerical Mathematics
  \textbf{46}(4), 813 (2006)

\bibitem{MR3126684}
A.~Archid, A.H. Bentbib, Approximation of the matrix exponential operator by a
  structure-preserving block {A}rnoldi-type method, Appl. Numer. Math.
  \textbf{75}, 37 (2014).
\newblock \doi{10.1016/j.apnum.2012.11.008}.
\newblock \urlprefix\url{https://doi.org/10.1016/j.apnum.2012.11.008}

\bibitem{benner98ans}
P.~Benner, V.~Mehrmann, H.~Xu, A numerically stable structure-preserving method
  for computing the eigenvalues of real {H}amiltonian or symplectic pencils,
  Numerische Mathematik \textbf{78}(3), 329 (1998)

\bibitem{arnoldi1951principle}
W.E. Arnoldi, The principle of minimized iterations in the solution of the
  matrix eigenvalue problem, Quarterly of applied mathematics \textbf{9}(1), 17
  (1951)

\bibitem{benner2011hamiltonian}
P.~Benner, H.~Fa{\ss}bender, M.~Stoll, A {H}amiltonian {K}rylov--{S}chur-type
  method based on the symplectic {L}anczos process, Linear Algebra and its
  Applications \textbf{435}(3), 578 (2011)

\bibitem{lall2003structure}
S.~Lall, P.~Krysl, J.E. Marsden, Structure-preserving model reduction for
  mechanical systems, Physica D: Nonlinear Phenomena \textbf{184}(1), 304
  (2003)

\bibitem{celledoni16epk}
E.~Celledoni, L.~Li, Energy-preserving {K}rylov projection methods for large
  and sparse linear {H}amiltonian systems, Proceedings of the ECMI conference
  \textbf{X}(X), YY (2016)

\bibitem{peng2016symplectic}
L.~Peng, K.~Mohseni, Symplectic model reduction of {H}amiltonian systems, SIAM
  Journal on Scientific Computing \textbf{38}(1), A1 (2016)

\end{thebibliography}

\section{Appendix}\label{sec:algorithm}
\begin{table}[H]
\vspace{-20pt}
\begin{subalgorithm}{.5\textwidth}
\begin{algorithmic}[1]
  \caption{ Arnoldi's algorithm with modified inner product} \label{alg:Arnoldi modified inner product algorithm}
\State Input: a matrix $J\in {\mathbb{R}}^{m\times m}$, $H\in {\mathbb{R}}^{m\times m}$, a vector $b\in\mathbb{R}^{m}$, a number $n\in\mathbb{N}$ and a tolerance $\iota\in\mathbb{R}$.

\State $A=JH$\;
\State $v_{1}=\frac{b}{\langle b,b\rangle_{H}^\frac{1}{2}}$\;
\For{$j=1 : n$}
\State compute $w_{j}=Av_{j}$\;
\For{$k=1: 2$}
\For{$i=1: j$}
\State $h_{i,j}=\langle v_{i},w_{j}\rangle_{H}$\;
\State $w_{j}=w_{j}-h_{i,j}v_{i}$\;
\EndFor
\EndFor
\State $h_{j+1,j}=\langle w_{j},w_{j}\rangle_{H}^\frac{1}{2}$\;
\If{$h_{j+1,j} < \iota$}
\State{Stop}
\EndIf
\State $v_{j+1}=w_{j}/h_{j+1,j}$\;
\EndFor
\State Output: $H_{n}, V_{n}, v_{n+1}, h_{n+1,n}$.
\end{algorithmic}
\end{subalgorithm}
\begin{subalgorithm}{.5\textwidth}
\begin{algorithmic}[1]
 \caption{ Algorithm to generate $V_{n}$ (by QR factorization)}\label{alg:orthogonal_Vn}
\State {Matrix $A\in {\mathbb{R}}^{2m\times 2m}$, vector $b\in {\mathbb{R}}^{2m}$, number $n\in\mathbb{N}$.}
\State $v=b$\;
\State $K_n=v$\;
\For{$i=1: n-1$}
\State $v=Av$\;
\State $K_n=[K_n,v]$\;
\EndFor
\State $K_n^q=K_n(1:m,:)$\;
\State $K_n^p=K_n(m+1:2m,:)$\;
\State $[Q,R]=qr([K_n^q,K_n^p])$\;
\State $V_{n}=Q(:,1:k),\quad k=\mathrm{rank}([K_n^q,K_n^p])\le 2n$\;
\State Output $V_{n}$.
\label{tab:1}
\end{algorithmic}
\end{subalgorithm}
\end{table}

\end{document}